\documentclass{amsart}

\usepackage{graphicx} 
\usepackage{amsmath}
\usepackage{amssymb}
\usepackage{amsfonts}
\usepackage{amsthm}
\usepackage{physics}
\usepackage{comment}
\usepackage[margin=29truemm]{geometry}
\usepackage{tikz}
\usepackage{enumitem}

\newcommand{\R}{\mathbb{R}}
\newcommand{\Z}{\mathbb{Z}}
\newcommand{\C}{\mathbb{C}}
\newcommand{\N}{\mathbb{N}}
\newcommand{\Q}{\mathbb{Q}}
\newcommand{\T}{\mathbb{T}}

\newcommand{\supp}{\mathrm{supp}\thinspace}
\newcommand{\F}{\mathcal{F}}

\newcommand{\limitsup}[2]{\underset{#1\to#2}{\limsup}\thinspace}

\newcommand{\vomega}{\vec{\omega}}

\usepackage{mathtools}

\theoremstyle{plain}
\newtheorem{theorem}{Theorem}
\newtheorem{corollary}{Corollary}

\newtheorem*{thm*}{Theorem}
\newtheorem*{lemma*}{Lemma}
\newtheorem{proposition}{Proposition}
\newtheorem{lemma}{Lemma}
\newtheorem{problem}{Problem}
\newtheorem{observation}{Observation}
\theoremstyle{definition}
\newtheorem{definition}{Definition}

\newtheorem{remark}{Remark}

\title[Strichartz estimates: An arithmetic approach]{Strichartz estimates for quasi-periodic functions on long time intervals: An arithmetic approach}
\author{Kotaro Inami}
\date{\today}


\begin{document}

\begin{abstract}
  We study long-time Strichartz estimates for the one-dimensional Schr\"{o}dinger equation with quasi-periodic initial data. For two-frequency data with an algebraic frequency ratio, we observe that the behavior of the linear Schr\"{o}dinger evolution changes depending on the algebraic degree of the ratio. Making use of this observation, we improve the Strichartz estimates on long time intervals. We also prove an endpoint $L^4$ Strichartz estimate. Our proofs use Roth-type Diophantine inequalities and Vinogradov-type mean value estimates for the Parsell--Vinogradov systems. 
\end{abstract}

\maketitle

\section{Introduction}
In this paper, we study Strichartz-type estimates for the Schr\"{o}dinger equation:
\begin{equation}
  2\pi i\partial_{t}u(x,t) = \partial_{xx} u(x,t) \quad (x,t)\in \R \times \R. \label{eq:Schrodinger}
\end{equation}
We investigate this equation with an initial datum $f$ that is neither periodic nor decaying at $\pm\infty$, for example, $f(x) = \sin(x) + \sin(\sqrt{2}x)$. We consider this problem in the framework of \textit{quasi-periodic functions}. For the definition of them, we first define the \textit{non-resonant} vectors: 
\begin{definition}
    We say that a vector $\vec{\omega} = (\omega_1, \omega_2, \cdots, \omega_\nu)\in \R^\nu$ is \textit{non-resonant} if $\omega_1, \omega_2, \cdots, \omega_\nu$ are linearly independent over $\Q$. 
\end{definition}
Formally, quasi-periodic functions are functions of the following forms: for some non-resonant vector $\vomega\in\R^\nu$ and some sequence $(a_{k})_{k\in\Z^\nu}\subset\C$, 
\begin{equation}
  f(x) = \sum_{k\in \Z^\nu}e^{2\pi i(k\cdot\vomega)x}a_{k}\quad (x\in \R). 
  \label{eq:quasi_periodic_series}
\end{equation}
We can easily see that the example $\sin(x) + \sin(\sqrt{2}x)$ is of this form. The convergence of these series is understood with respect to certain $L^p$-type topologies. To make this precise, we introduce the set of quasi-periodic trigonometric polynomials.
\begin{definition}
    For non-resonant $\vec{\omega}\in \R^\nu$, we define the set of 
    trigonometric polynomials $\mathcal{T}_{\vec{\omega}}(\R)$ as
    \begin{equation*}
        \mathcal{T}_{\vec{\omega}}(\R):=\Bigl\{\sum_{k\in \Z^\nu}a_k e^{2\pi i(\vec{\omega} \cdot k)x} \;;\; only\; finitely\; many \; 
        coefficients \; a_k\;are \;nonzero\Bigr\}. 
    \end{equation*}
\end{definition}

Now, for $p\in [1,\infty)$ and for trigonometric polynomial $f\in \mathcal{T}_{\vomega}(\R)$, we define a norm $\|\cdot\|_{\mathcal{L}^p(\R)}$ by
\begin{equation*}
    \|f\|^p_{\mathcal{L}^p(\R)} := \limitsup{L}{\infty}\frac{1}{2L}\int^{L}_{-L}|f(x)|^pdx. 
\end{equation*}
For $p = \infty$, we define $\|\cdot\|_{\mathcal{L}^{\infty}}$ by $\|\cdot\|_{L^{\infty}}$. In this paper, 
we work with the spaces of quasi-periodic functions $\mathcal{B}^{p}_{\vec{\omega}}(\R)\; (p\in [1,\infty))$ 
as the completions of $\mathcal{T}_{\vec{\omega}}(\R)$ under the above norm for 
$p\in [1,\infty]$. 

The Fourier transform of a trigonometric polynomial is defined by 
\[
\widehat{f}(\lambda) = \F f(\lambda) := \lim_{L\to\infty}\frac{1}{2L}\int^{L}_{-L}f(x)e^{-2\pi i\lambda x} dx\quad (\lambda = k\cdot \vec{\omega}, \; k\in \Z^\nu)
\]
and the inverse Fourier transform is defined by 
\[
\F ^{-1}f(x) := \sum_{\substack{\lambda = k\cdot\vec{\omega}\\k\in \Z^\nu}}a_{\lambda}e^{2\pi i\lambda x}\quad (x\in \R). 
\]
These transforms are continuously extended to $\mathcal{B}^{p}_{\vec{\omega}}(\R)$. The solution to the Schr\"{o}dinger equation \eqref{eq:Schrodinger} with an initial datum $f\in \mathcal{B}^{2}_{\vec{\omega}}(\R)$ is formally written via the Fourier transform as 
\[e^{2\pi itD^{2}_{x}}f(x) 
= \sum_{k\in \Z^{\nu}} 
e^{2\pi i [(k\cdot \vec{\omega})x + |k\cdot \vec{\omega}|^2 t]}
\widehat{f}(k\cdot \vec{\omega}) \]
where $D_{x} := \frac{1}{2\pi i}\partial_{x}$. 

Now, we review some related works. Oh~\cite{Oh} proved local well-posedness for power-type nonlinear Schr\"{o}dinger equations with almost periodic initial data whose Fourier coefficients are absolutely summable; in particular, his result applies to our quasi-periodic initial data with $\ell^1$-summable Fourier coefficients. Furthermore, quasi-periodic Cauchy problems for nonlinear Schr\"{o}dinger equations \cite{Papenburg, Xu,damanik2024,damanik20242}, the KdV equation \cite{Tsugawa}, the Benjamin--Ono equation \cite{aitzhan2024}, and other nonlinear equations \cite{zhao2024,Ifrim} have been also studied. 

For the Strichartz estimates, Klaus~\cite{klaus2023} first introduced an estimate of this kind for Schr\"{o}dinger equation: 
\[
\|e^{2\pi itD^{2}_{x}}f(x)\|_{\mathcal{L}^{4}_{t,x}(\R^2)}\lesssim \|f(x)\|_{\mathcal{L}^2_{x}(\R)}
\]
where $\|g(t,x)\|^4_{\mathcal{L}^{4}_{t,x}(\R^2)}:= \limsup_{T\to\infty}\frac{1}{2T}\int^{T}_{-T}\limsup_{L\to\infty}\frac{1}{2L}\int^{L}_{-L}|g(t,x)|^4dxdt$. 
However, Klaus explained that this estimate does not imply the estimate 
\[
\Bigl\|\int^{T}_{-T}e^{2\pi i(t -s)D^{2}_{x}}F(s,x)ds\Bigr\|_{\mathcal{L}^{4}_{t,x}}\lesssim T\|F\|_{\mathcal{L}^{4/3}_{t,x}}
\] 
because the time variable is also averaged. Thus, this type of estimate is hard to apply to show the local well-posedness results for nonlinear equations. Schippa~\cite{Schippa_quasiperiodic} proved a different type of Strichartz estimates via Bourgain--Demeter's decoupling estimate \cite{Bourgain_Demeter_decoupling} or C\'{o}rdoba--Fefferman's reverse square function estimate \cite{Cordoba,Fefferman}. To state his results, we need to define the Sobolev-type spaces $\mathcal{H}^s_{\vec{\omega}}(\R)$ for quasi-periodic functions and a suitable decomposition in the frequency space: 
\begin{definition}
    Let $\vec{\omega} \in \R^\nu$ be non-resonant, $2\le \nu\in \N$, and $s\in\R$. For $f\in \mathcal{T}_{\vec{\omega}}(\R)$, 
    we define the norm by
    \[\|f\|^2_{\mathcal{H}^s_{\vec{\omega}}(\R)}:= \sum_{k\in \Z^\nu} (1 + |k|^2)^{s}|\widehat{f}(k\cdot \vec{\omega})|^2.\]
    The Sobolev-type space $\mathcal{H}^s_{\vec{\omega}}(\R)$ is defined as 
    the completion of $\mathcal{T}_{\vec{\omega}}(\R)$ under $\|\cdot\|_{\mathcal{H}^s_{\vec{\omega}}}$. 
\end{definition}

As pointed out by Schippa in \cite{Schippa_quasiperiodic}, there are two possible Littlewood--Paley-type projections for quasi-periodic functions. 
\begin{definition}
    Let $p\in [1,\infty)$ and $N, C > 0$. For any trigonometric polynomial 
    $f\in \mathcal{T}_{\vec{\omega}}(\R)$, the \textit{Frequency projection} $P_{N}$ is given by 
    \[P_N f(x) := \sum_{\substack{\lambda = k\cdot \vec{\omega}\\ k\in\Z^\nu, N/2 \leq|\lambda| < N}}\widehat{f}(\lambda)e^{2\pi i \lambda x}\quad (x\in \R) \]
    and the \textit{height projection} $R_{C}$ is given by 
    \[R_C f(x) = \sum_{\substack{\lambda = k \cdot \vec{\omega}\\ k\in \Z^\nu, C / 2\le |k| < C}}\widehat{f}(\lambda)e^{2\pi i \lambda x} \quad (x\in \R). \]
\end{definition}

Using the height projections $R_{C}$, the Sobolev-type norm $\|\cdot\|_{\mathcal{H}^{s}_{\vomega}}$ can be written as 
\[
\|f\|^2_{\mathcal{H}^{s}_{\vomega}(\R)}\sim \sum_{C:dyadic}C^{2s}\|R_{C}f\|^{2}_{\mathcal{L}^2(\R)}. 
\]

\begin{remark}
Note that the summation defining frequency projection $P_N$ is an infinite sum, whereas that defining $R_{C}$ is a finite sum. 
\end{remark}
Schippa's results are as follows: 
\begin{theorem}[Theorem 1.1. and the $1d$ version of Theorem 3.1. in \cite{Schippa_quasiperiodic}]
    Let $2\le \nu \in \N$, $p\in (2,\infty)$ and let $\vec{\omega}\in \R^\nu$ be non-resonant. 
    Then, for any $f\in \mathcal{H}^{s}_{\vec{\omega}}$, the following estimates hold: 
    \begin{equation}
      \|e^{2\pi itD^{2}_{x}}f(x)\|_{L^{4}_{t}\mathcal{L}^{4}_{x}([-1,1]\times \R)} \lesssim \|f\|_{\mathcal{H}^s_{\vec{\omega}}(\R)}\label{eq:Schippa_l4}
    \end{equation}
    with $s > \frac{\nu - 1}{4}$ and 
    \begin{equation}
      \|e^{2\pi itD^{2}_{x}}f(x)\|_{L^{p}_{t}\mathcal{L}^{p}_{x}([-1,1]\times \R)} \lesssim \|f\|_{\mathcal{H}^s_{\vec{\omega}}(\R)}\label{eq:Schippa_l6}
    \end{equation}
    with $s > (\nu - 1)\Bigl(\frac{1}{2} - \frac{1}{p}\Bigr) + s_{p}$
    where 
    \begin{align*}
      s_{p} := \begin{cases}
        0 &(2\leq p \leq 6),\\
        \frac{1}{2} - \frac{3}{p} &(p > 6). 
      \end{cases}
    \end{align*}
    \label{theorem:schippa_decoupling}
\end{theorem}
Using the height projections $R_{C}$, these estimates can be written as
\[
\|e^{2\pi itD^{2}_{x}}R_{C}f(x)\|_{L^{4}_{t}\mathcal{L}^{4}_{x}([-1,1]\times \R)} \lesssim C^{\frac{\nu - 1}{4}}\|R_{C}f\|_{\mathcal{L}^{2}(\R)}
\]
and 
\[
\|e^{2\pi itD^{2}_{x}}R_{C}f(x)\|_{L^{p}_{t}\mathcal{L}^{p}_{x}([-1,1]\times \R)} \lesssim C^{(\nu - 1)\bigl(\frac{1}{2} - \frac{1}{p}\bigr) + s_{p}}\|R_{C}f\|_{\mathcal{L}^{2}(\R)}. 
\]
The first inequality \eqref{eq:Schippa_l4} was proved via C\'{o}rdoba--Fefferman's reverse square function estimate while the second one was proved via Bourgain--Demeter's decoupling estimate. Hence, the loss $s_{p}$ in \eqref{eq:Schippa_l6} comes from the decoupling inequality. 
The factors $C^{\frac{\nu -1}{4}}$ and $C^{(\nu -1 )(1/2 - 1/p)}$ come from the bound for the number of lattice points that are contained in 
\[\{k\in \Z^\nu \;  ; \; |k \cdot \vec{\omega}| \in [0,1], \; |k|\sim C\}. \]
The cardinality of this set is bounded by $\lesssim C^{\nu -1}$ because the set 
can be covered by a rectangle with scale 
$1 \times \underbrace{C \times \cdots \times C}_{\nu -1}$ 
and the number of lattice points contained in the rectangle is less than $C^{\nu -1}$. 
In \cite{Schippa_quasiperiodic}, this estimate was proved to be sharp up to 
endpoints when $(\nu, d, p) = (2, 1, 4), (2,2,4), (2,1,6)$. 

In this paper, we further investigate Schippa-type Strichartz estimates. First of all, we consider the following problem:
\begin{problem}
  Let $C \geq 1$ and $T \geq 1$. Suppose that $\omega\in \R\backslash\Q$ and $\vomega = (1, \omega)\in\R^2$. Then what is the best constant for the inequality
  \[
  \|e^{2\pi itD^{2}_{x}}R_{C}f(x)\|_{L^{p}_{t}\mathcal{L}^{p}_{x}([0,T]\times \R)} \lesssim ST(T,C,p)\|R_{C}f\|_{\mathcal{L}^{2}(\R)}? 
  \]
\end{problem}
Trivially, from Theorem \ref{theorem:schippa_decoupling}, it follows that
 \[\|e^{2\pi itD^{2}_{x}}R_Cf(x)\|_{L^{p}_{t}\mathcal{L}^{p}_{x}([-T,T]\times \R)} \lesssim T^{\frac 1p} C^{s}\|R_Cf\|_{\mathcal{L}^2_{x}(\R)}\]
for $\nu = 2$, $p\geq 6$, $T\ge 1$ and $s > (1/2 - 1/p) + s_{p}$. Thus, we shall ask whether Strichartz estimates with constants that are smaller than $T^{\frac 1p}C^s$ can hold. 

To motivate this problem, we present a heuristic observation that is similar to the one in \cite{Deng_Germain_Guth_irrational} for the Schr\"{o}dinger equation on irrational tori: 
\begin{observation}
    Let $\nu = 2$ and $\vec{\omega} = (1,\sqrt{2})$. Then
    the solution $e^{2\pi itD^{2}_{x}}R_C f(x)$ is "almost" periodic in time with
    periodicity near $C^{2}$. 
\end{observation}
Fix $\varepsilon>0$. We shall use the classical Diophantine approximation
of $\sqrt{2}$ to explain this observation: there exist $p\in \Z$ and $q\in \N$ with $q\sim C^{2 + \varepsilon}$
 such that $\Bigl|q\sqrt{2} - p \Bigr| \leq q^{-1}\sim  C^{- 2 - \varepsilon}$. For $x\in \R$, let $\|x\|$ denote
 the distance to the closest integer. Then, for $k_1, k_2\in \Z$ with $|k_1|, |k_2| \sim C$, 
 \begin{align*}
    \|q|k_1 + k_2 \sqrt{2}|^2\| &= \|q[(k^2_1 + k^2_2) + 2k_1 k_2 \sqrt{2}]\|\\
    &= \|2qk_1 k_2 \sqrt{2}\|\\
    &\leq |2k_1 k_2 (q\sqrt{2} - p)|\\
    &\lesssim C^{2}C^{-2 - \varepsilon} = C^{-\varepsilon} \ll 1
 \end{align*}
 for sufficiently large $C > 0$. This suggests that, when $t \sim C^{2 + \varepsilon}$
 the phase $t\lambda^2$ is very close to an integer. Thus, we heuristically obtain
 \[e^{2\pi itD^{2}_{x}} R_Cf(x) \sim R_Cf(x)\]
 when $t \sim C^{2 + \varepsilon}$.  

This observation suggests the following claim: for $T > C^{2 + \varepsilon}$ and $p > 6$, we divide 
 the time interval $[0,T]$ into $T / C^{2 + \varepsilon}$ subintervals of length $\sim C^{2 + \varepsilon}$ and 
 formally apply the Strichartz estimate \eqref{eq:Schippa_l6} on each subinterval, then we expect that
 \[\|e^{2\pi it D^{2}_{x}}R_C f(x)\|_{L^{p}_{t}\mathcal{L}^{p}_{x}([0,T]\times\R)} \lesssim_{\varepsilon} \Bigl(\frac{T}{C^{2 + \varepsilon}}\Bigr)^{\frac 1p} C^{1 - \frac 4p + \varepsilon}\|R_C f(x)\|_{\mathcal{L}^{2}_{x}(\R)} . \]

We first state the result when $\omega$ is an algebraic real number of degree $2$. In this case, we attain the estimate
with the same power loss as in the observation. 
\begin{theorem}
    Let $\varepsilon > 0$, $\nu = 2$, and let $\omega\in \R$ be an algebraic irrational number of degree $2$. Suppose that $\vec{\omega} = (1, \omega)$.  
    Then, for $p \geq 6$ and $T > C^{2 + \varepsilon}$, we have
    \begin{equation}
        \|e^{2\pi i tD^{2}_{x}}R_Cf(x)\|_{L^p_{t}\mathcal{L}^{p}_{x}([0,T]\times \R)} 
        \lesssim_{\varepsilon} \Bigl\{1 +\Bigl(\frac{T}{C^{2}}\Bigr)^{\frac 1p}\Bigr\}
        C^{1 - \frac{4}{p} + \varepsilon}\|R_Cf(x)\|_{\mathcal{L}^{2}_{x}(\R)}
        \label{eq:main_thm_degree2}
    \end{equation}
    \label{theorem:long_time_Strchartz}
\end{theorem}

Next, we ask what happens if $\omega$ has degree greater than $2$. We demonstrate a heuristic observation
with $\vec{\omega} = (1, 2^{\frac 13})$.
In this case, the "period" in time $t$ might be different 
from the case of degree $2$ due to the independence of $1, 2^{\frac 13}, 2^{\frac 23}$. 

\begin{observation}
  Let $\nu = 2$ and $\vec{\omega} = (1, 2^{\frac 13})$. Then the solution
  $e^{2\pi i tD^{2}_{x}}R_{C}f(x)$ is "almost" periodic in time with periodicity
  near $C^4$
\end{observation}
Fix $\varepsilon>0$. For $k_1, k_2\in \Z$ with $|k_1|, |k_2| \sim C$, it follows that
 \[|k_1 + k_2 2^{\frac 13}|^2 
 = k^2_1 + (2k_1 k_2) 2^{\frac 13} + k^2_2 2^{\frac 23}. \]
In this case, we have to approximate
$2^{\frac 13}$ and $2^{\frac 23}$ simultaneously. 
This leads us to use the next Diophantine approximation: 
there exist $p_1, p_2\in \Z$ 
and $q\in \Z$ with $C^{4-\varepsilon}\lesssim q \lesssim C^{4}$ such that 
\[|q2^{\frac 13} - p_1|, |q 2^{\frac 23} - p_2| \leq C^{-2 - \frac{\varepsilon}{2}}. \]
From this approximation, we obtain
\begin{align*}
  \|q(k_1 + 2^{\frac 13}k_2)^2\| &\leq |qk^2_2 2^{\frac 23} - k^2_2 p_1| 
  + |q(2k_1 k_2)2^{\frac 13} - (2k_1 k_2)p_2|\\
  &\lesssim C^{2}\{|q 2^{\frac 23} - p_1| + |q 2^{\frac 13} - p_2|\}\\
  &\leq C^{-\frac{\varepsilon}{2}} \ll 1
\end{align*}
for sufficiently large $C \geq 1$. We expect that the solution refocuses 
the initial data around $t \sim C^{4}$ and expect 
\[
\|e^{2\pi it D^{2}_{x}}R_C f(x)\|_{L^{p}_{t}\mathcal{L}^{p}_{x}([0,T]\times\R)} \lesssim_{\varepsilon} \Bigl(\frac{T}{C^{4 + \varepsilon}}\Bigr)^{\frac 1p} C^{1 - \frac 4p + \varepsilon}\|R_C f(x)\|_{\mathcal{L}^{2}_{x}(\R)}. 
\]
Our corresponding result for this case is the following. 
\begin{theorem}
  Let $\varepsilon > 0$, $\nu = 2$, and let $\omega\in \R$ be an algebraic irrational number of degree greater than $2$. Suppose that $\vec{\omega} = (1, \omega)$. 
    Then, for $6 \leq p < 14$ and $T > C^{2}$, we have
    \begin{equation}
      \|e^{2\pi i tD^{2}_{x}}R_Cf(x)\|_{L^p_{t}\mathcal{L}^{p}_{x}
      ([0,T]\times \R)} \lesssim_{\varepsilon} ST(T,C,p, \varepsilon)
      \|R_Cf(x)\|_{\mathcal{L}^{2}_{x}(\R)} \label{eq:main_thm_degree3}
    \end{equation}
    where 
    \begin{align*}
    ST(T,C,p, \varepsilon)=
    \begin{cases}
      \Bigl\{1 +\Bigl(\frac{T}{C^{2}}\Bigr)^{\frac 1p}\Bigr\}
      C^{1 - \frac{2}{p} + \varepsilon} \quad &(C^2 < T < C^4),\\
      \Bigl\{1 +\Bigl(\frac{T}{C^{4}}\Bigr)^{\frac 1p}\Bigr\}
      C^{1 - \frac{2}{p} + \varepsilon} &(T \geq C^4). 
    \end{cases}
    \end{align*}
    Moreover, if $p \geq 14$ and $T\geq 1$, then 
    \begin{equation}
      \|e^{2\pi i tD^{2}_{x}}R_Cf(x)\|_{L^p_{t}\mathcal{L}^{p}_{x}
        ([0,T]\times \R)} \lesssim_{\varepsilon} C^{\varepsilon}(C^{1 - \frac{4}{p}} + C^{1-\frac{8}{p}}T^{\frac{1}{p}})
        \|R_Cf(x)\|_{\mathcal{L}^{2}_{x}(\R)}
        \label{eq:main_thm_degree3_l14}
    \end{equation}
    holds. 
    \label{theorem:degree3}
\end{theorem}
Note that, in the regime of $C^2 \leq T < C^4$, the estimate \eqref{eq:main_thm_degree3} is the same as that in
Theorem \ref{theorem:schippa_decoupling} with our assumptions. In contrast, 
when $T \geq C^4$, the constant becomes smaller and has the same scaling as in
Theorem \ref{theorem:long_time_Strchartz}. This suggests that the phenomenon in Observation 2 might occur in this regime. Substituting $f(x) = e^{2\pi i(k\cdot\vomega)x}$ with $|k|\sim C$, we find that our estimate \eqref{eq:main_thm_degree3} is essentially
sharp when $p = 6$ and $T \sim C^4$. Also \eqref{eq:main_thm_degree3_l14} is consistent with Observation 2. However, we do not know whether the estimate can be improved in the intermediate range $C^2 \leq T < C^4$ and whether for $6 < p < 14$, the estimate is sharp.

We note a similarity with Strichartz estimates on an irrational torus. To see this, we utilize the Kronecker--Weyl type identity, which will be shown in Section 2:
\begin{lemma*}[Kronecker-Weyl]
  Let $\nu \geq 2$, $\vec{\omega}\in\R^\nu$ be non-resonant, and $p\in [1,\infty)$. 
  Then for any trigonometric polynomial 
  $\sum_{k\in \Lambda}a_{k}e^{2\pi i (\vec{\omega}\cdot k)x}\in 
  \mathcal{T}_{\vec{\omega}}(\R)$ ($\Lambda\subset \Z^\nu$ is a finite subset), 
  we have 
  \[\Bigl\|\sum_{k\in\Lambda}a_{k}e^{2\pi i (\vec{\omega}\cdot k)x}\Bigr\|_{\mathcal{L}^{p}_{x}(\R)} 
  = \Bigl\|\sum_{k\in \Lambda}a_{k}e^{2\pi i k\cdot y}\Bigr\|_{L^p_{y}(\T^\nu)}\]
\end{lemma*}
Let $\vomega = (1,\omega)$ and $\omega\in\R\backslash\Q$. 
From this identity, our Strichartz estimate 
\[
 \|e^{2\pi i tD^{2}_{x}}R_Cf(x)\|_{L^p_{t}\mathcal{L}^{p}_{x}([0,T]\times \R)} \lesssim ST(T,C)\|R_Cf(x)\|_{\mathcal{L}^{2}_{x}(\R)}
\]
for some $p\in [2,\infty)$ and $s \geq 0$
can be viewed as an inequality on the two-dimensional torus:
\[
\|e^{2\pi i t\tilde{\Delta}}P_{C}f(x)\|_{L^{p}_{t,x}([0,T]\times \T^{2})}\lesssim ST(T,C)\|P_{C}f(x)\|_{L^{2}_{x}(\T^2)}
\]
where $\tilde{\Delta}:= \partial^{2}_{x_{1}} + 2\omega\partial_{x_{1}}\partial_{x_2} + \omega^{2}\partial^2_{x_2}$ and $P_{C}f(x) = \sum_{|k|\sim C}a_{k}e^{2\pi ik\cdot x}$. The symbol of our Fourier multiplier $\tilde{\Delta}$ is written as 
\[
-4\pi^2 (k_1, k_2)\begin{pmatrix}
  1 & \omega\\
  \omega & \omega^{2}
\end{pmatrix}
\begin{pmatrix}
  k_1\\
  k_2
\end{pmatrix}. 
\]
We let $A_{\omega}$ denote the $2\times 2$ matrix inside the symbol. From this observation, we obtain a corollary of Theorem~\ref{theorem:long_time_Strchartz} and Theorem~\ref{theorem:degree3}: 
\begin{corollary}
  Let $\varepsilon > 0$ and $C \geq 1$. Depending on $T > 0$, $p\geq 2$, and $\omega\in \R\backslash\Q$, the following estimates hold for any smooth functions $f$ on $\T^2$: 
  \begin{enumerate}
    \item If $\omega\in\R\backslash\Q$ is an algebraic number of degree $2$, $T \geq C^{2 + \varepsilon}$, and $p \geq 6$, 
    \[
    \|e^{2\pi i t\tilde{\Delta}}P_{C}f(x)\|_{L^{p}_{t,x}([0,T]\times \T^{2})}\lesssim_{\varepsilon}C^{\varepsilon}(1 + T^{\frac 1p}C^{-\frac{2}{p}})C^{1 - \frac 4p}\|P_{C}f(x)\|_{L^{2}_{x}(\T^2)}. 
    \]
    \item If $\omega\in\R\backslash\Q$ is an algebraic number of degree greater than $2$, $T \geq C^{4 + \varepsilon}$, and $p \geq 6$, 
    \[
    \|e^{2\pi i t\tilde{\Delta}}P_{C}f(x)\|_{L^{p}_{t,x}([0,T]\times \T^{2})}\lesssim_{\varepsilon}C^{\varepsilon}(1 + T^{\frac 1p}C^{-\frac{4}{p}})C^{1 - \frac 2p}\|P_{C}f(x)\|_{L^{2}_{x}(\T^2)}. 
    \]
    \item If $\omega\in\R\backslash\Q$ is an algebraic number of degree greater than $2$, $T \geq 1$, and $p \geq 14$, 
    \[
    \|e^{2\pi i t\tilde{\Delta}}P_{C}f(x)\|_{L^{p}_{t,x}([0,T]\times \T^{2})}\lesssim_{\varepsilon}C^{\varepsilon}(C^{1 - \frac{4}{p}} + C^{1-\frac{8}{p}}T^{\frac{1}{p}})\|P_{C}f(x)\|_{L^{2}_{x}(\T^2)}.
    \]
  \end{enumerate}
  \label{cor:main}
\end{corollary}

In Deng--Germain--Guth--Myerson~\cite{Degn_Germain_Guth_Myerson}, the Strichartz estimates on non-rectangular irrational tori were considered. After a suitable change of variables, the estimates are reduced to the inequalities on a rectangular torus with a twisted Laplacian $\Delta_{A}$: 
\[
\|e^{2\pi i t\Delta_{A}}P_{C}f(x)\|_{L^{p}_{t,x}([0,T]\times \T^{2})}\lesssim ST(T,C)\|P_{C}f(x)\|_{L^{2}_{x}(\T^2)}
\]
where the symbol of the differential operator $\Delta_{A}$ is given by
\[
-4\pi^2 (k_1, k_2)A
\begin{pmatrix}
  k_1\\
  k_2
\end{pmatrix}
\]
for some symmetric and positive definite matrix $A$. So, our problem can also be regarded as the degenerate version ($\det A_{\omega} = 0$) of their problems. Deng--Germain--Guth--Myerson~\cite{Degn_Germain_Guth_Myerson} proved the following estimate: 
\begin{theorem}[The two dimensional version of Theorem 1.1. and Theorem 1.2. in \cite{Degn_Germain_Guth_Myerson}]
  For $p > 4$, any symmetric, positive definite matrices $A$ with some Diophantine conditions\footnote{In their paper, they considered all $A$ that belong to the complement of a set of Lebesgue measure zero. }, any $\varepsilon > 0$, and any smooth functions on $\T^2$, 
  \begin{equation*}
    \|e^{2\pi it\Delta_{A}}P_{C}f\|_{L^{p}([0,T]\times \T^2)}\lesssim_{\varepsilon} C^{\varepsilon}
    \Bigl[C^{1 - \frac 4p} +T^\frac{1}{p}C^{\frac{2}{3}\bigl(1 - \frac 4p\bigr)} + C^{\frac{1}{2} - \frac{3}{p}} T^\frac{1}{p} + C^{1 - \frac 8p} T^\frac{1}{p}\Bigr]
    \|P_{C}f\|_{L^{2}(\T^2)}
  \end{equation*}
  \label{thm:DGGM}
\end{theorem}
In Deng--Germain--Guth~\cite{Deng_Germain_Guth_irrational}, the diagonal case was also considered: 
\begin{theorem}[Two dimensional version of Theorem 3.1. in \cite{Deng_Germain_Guth_irrational}]
  For any $\varepsilon >  0$ and diagonal positive definite $A$ with some Diophantine type assumptions, the following estimate holds for $p > 6$ and for any smooth functions $f$ on $\T^2$: 
  \[
  \|e^{2\pi i t\Delta_{A}}P_{C}f\|_{L^{p}([0,T]\times \T^2)}\lesssim_{\varepsilon}C^{\varepsilon}\bigl[C^{1 -\frac{4}{p}} + T^{\frac{1}{p}}C^{1 - \frac{6}{p}}\bigr]\|P_{C}f\|_{L^{2}(\T^2)}. 
  \]
  \label{thm:diagonal}
\end{theorem}
Comparing these results with Corollary~\ref{cor:main}, for the case when $\omega$ is of degree $2$ and $T\geq C^2$, our exponents match those in Theorem~\ref{thm:diagonal} and for the case when $\omega$ is of degree greater than $2$ and $p\geq 14$, our exponents match Theorem~\ref{thm:DGGM}.

Finally, we also show the endpoint case of the $L^4$ Strichartz estimate \eqref{eq:Schippa_l4}. The key ingredient for this result is a Littlewood--Paley type estimate for the height projections $R_{C}$ (see Lemma~\ref{lem:littlewood_paley}). We state this result for the fractional Schr\"{o}dinger propagators
\[e^{2\pi i t|D_{x}|^{a}}f(x)
     := \sum_{k\in \Z^{\nu}}
    e^{2\pi i[(\vec{\omega}\cdot k)x + |\vec{\omega}\cdot k|^a t]
    }\widehat{f}(\vec{\omega}\cdot k) \quad (f\in \mathcal{T}_{\vomega}(\R)).  \]
\begin{theorem}
  Let $\nu \geq 2$,
  $\vec{\omega}\in \R^{\nu}$ be non-resonant and $a\in (0,1)\cup(1,\infty)$. 
  Then, the estimate
  \begin{equation}
    \|e^{2\pi i t|D_{x}|^{a}}f(x)\|_{L^{4}_{t}\mathcal{L}^{4}_{x}
    ([0,1]\times\R)}
    \lesssim \|f(x)\|_{\mathcal{H}^{\frac{\nu - 1}{4} + \sigma_{a}}_{\vomega}(\R)} \label{eq:endpoint_schippa}
  \end{equation}
  holds for 
  \begin{align*}
    \sigma_a = \begin{cases}
      0\quad &(a\geq 2),\\
      \frac{2 - a}{8} & (a\in (0,1)\cup (1,2)).
    \end{cases}
  \end{align*}
  \label{thm:endpoint_schippa}
\end{theorem}

The organization of this paper is as follows: In Section 2, we recall useful results from analytic number theory and prepare some basic estimates. In Section 3, some examples for the Strichartz estimate are given. In Section 4, Section 5, and Section 6, we prove Theorem~\ref{theorem:long_time_Strchartz}, Theorem~\ref{theorem:degree3}, and Theorem~\ref{thm:endpoint_schippa} respectively. 

\subsection*{Notations}
\begin{itemize}
  \item $A\lesssim B$ (resp. $A\gtrsim B$) denotes the inequality $A \leq cB$ (resp. $A\geq cB$) with a harmless constant. If $A\lesssim B$ and $A\gtrsim B$ at the same time, we write $A\sim B$. 
  \item $A\lesssim_{\varepsilon}B$ denotes the inequality $A \leq c(\varepsilon)B$ with a harmless constant depending on the parameter $\varepsilon$. 
  \item For $1\leq p < \infty$ and $A\subset\R$, $L^{p}$-norm for Lebesgue measurable functions $f:A\to\C$ is defined as 
  \[
  \|f\|_{L^p (A)}:= \Bigl(\int_{A}|f(x)|^p dx\Bigr)^{\frac{1}{p}}
  \]
  and for $p = \infty$, 
  \[
  \|f\|_{L^{\infty}(A)}:= \underset{x\in A}{\mathrm{ess.sup}}|f(x)|
  \]
  We may write $\|f(x)\|_{L^{p}_{x}}$ to denote them. We use the following notation to write the mixed-norm: 
  \[
  \|f(t,x)\|_{L^{p}_{t,x}}:= \|\|f(t,x)\|_{L^{p}_{t}}\|_{L^{p}_{x}}. 
  \]
\end{itemize}

\section{Preliminaries}
\subsection{Tools from analytic number theory}
We repeatedly use Roth's theorem and its multivariate extension; see, for example, Chapter 7 of \cite{Bombieri_Gubler}. 
\begin{theorem}[Roth's theorem]
  Let $\alpha$ be an irrational algebraic number. Then, for every $\varepsilon > 0$
  there exists $C(\alpha,\varepsilon) > 0$ such that 
  every pair of integers $(p,q)\in\Z^2$ with $q\neq 0$ satisfies    
  \[\Bigl|\alpha - \frac pq \Bigr| \geq \frac{C(\alpha,\varepsilon)}{|q|^{2 + \varepsilon}}.\]
  \label{theorem:Roth}
\end{theorem}

\begin{theorem}
  Suppose $\alpha_1, \cdots, \alpha_{n}\in \R$ are algebraic real numbers and
  are linearly independent over $\Q$. For every $\varepsilon > 0$, there exists
  a constant $C(\alpha_1, \cdots, \alpha_n, \varepsilon)$ such that
  for any nonzero integer tuple $(x_1, \cdots, x_n)\in \Z^n$, we have
  \begin{equation}
    |\alpha_1 x_1 + \cdots + \alpha_n x_n| 
    \geq C(\alpha_1, \cdots, \alpha_n, \varepsilon)
    \max\{1, |x_1|, \cdots, |x_n|\}^{1-n-\varepsilon}
    \label{eq:multi_roth}
  \end{equation}
  \label{theorem:multi_roth}
\end{theorem}

In the proof of Theorem~\ref{theorem:long_time_Strchartz} and Theorem~\ref{theorem:degree3}, we need to count the number of lattice points on certain ovals. The following result gives a useful bound for this purpose:
\begin{theorem}[Theorem 3 in \cite{BP_counting_integers}]
  Suppose $\phi:\mathbb{S}^1\to\R^2$ is analytic. Then for all $\varepsilon > 0$, 
  \[
  \#(r\phi(\mathbb{S}^1)\cap \Z^2)\lesssim_{\varepsilon,\phi}r^{\varepsilon}. 
  \]
  \label{thm:BP_counting}
\end{theorem}

\subsection{Basic estimates}
Throughout this subsection, the following Kronecker-Weyl type identity plays an 
important role. 
\begin{lemma}[Kronecker-Weyl]
  Let $\nu \geq 2$, $\vec{\omega}\in\R^\nu$ be non-resonant, and $p\in [1,\infty)$. 
  Then for any trigonometric polynomial 
  $\sum_{k\in \Lambda}a_{k}e^{2\pi i (\vec{\omega}\cdot k)x}\in 
  \mathcal{T}_{\vec{\omega}}(\R)$ ($\Lambda\subset \Z^\nu$ is a finite subset), 
  we have 
  \[\Bigl\|\sum_{k\in\Lambda}a_{k}e^{2\pi i (\vec{\omega}\cdot k)x}\Bigr\|_{\mathcal{L}^{p}_{x}(\R)} 
  = \Bigl\|\sum_{k\in \Lambda}a_{k}e^{2\pi i k\cdot y}\Bigr\|_{L^p_{y}(\T^\nu)}\]
  \label{lemma:kronecker_weyl}
\end{lemma}

\begin{proof}
  Note that for 
  $\Bigl|\sum_{k\in\Lambda}a_{k}e^{2\pi i (\vec{\omega}\cdot k)x}\Bigr|^p
  = \Bigl|\sum_{k\in\Lambda}a_{k}e^{2\pi i (\vec{\omega} x)\cdot k}\Bigr|^p$, 
  there exists a continuous function $F\in C(\T^\nu)$ such that 
  \[\Bigl|\sum_{k\in\Lambda}a_{k}e^{2\pi i (\vec{\omega} x)\cdot k}\Bigr|^p
   = F(\omega_1 x, \cdots, \omega_\nu x). \]
   From the Weierstrass approximation theorem, for any $\varepsilon > 0$, there
   exists finite number of coefficients $b_{k}$ such that 
   \[\sup_{y\in \T^\nu}\Bigl|F(y) - \sum_{k}b_{k}e^{2\pi i k\cdot y}\Bigr| 
   < \varepsilon. \]
   Hence, by the standard argument, it is enough to show 
   \[\limitsup{L}{\infty}\frac{1}{2L} 
   \int^{L}_{-L}\sum_{k}b_{k}e^{2\pi i (\vec{\omega}\cdot k)x}dx
   = \int_{\T^\nu}\sum_{k}b_{k}e^{2\pi i k\cdot y}dy. \]
   Since $\vec{\omega}\in \R^\nu$ is non-resonant, we have
  \begin{align*}
  \limsup_{L\to \infty}\frac{1}{2L}\int^{L}_{-L}e^{2\pi i(w\cdot k)x}dx = 
  \begin{cases}
    1 &\quad (k =0),\\
    0 &\quad (k\neq 0). 
  \end{cases}
\end{align*}
Hence, it follows that 
\[\limitsup{L}{\infty}\frac{1}{2L} 
   \int^{L}_{-L}\sum_{k}b_{k}e^{2\pi i (\vec{\omega}\cdot k)x}dx = b_{0}. \]
On the other hand, we also have
\begin{align*}
  \int_{\T^\nu}e^{2\pi i k\cdot y}dy = 
  \begin{cases}
    1 &\quad (k =0),\\
    0 &\quad (k\neq 0). 
  \end{cases}
\end{align*}
Therefore, it holds that 
\[\int_{\T^\nu}\sum_{k}b_{k}e^{2\pi i k\cdot y}dy = b_0.\]
This completes the proof. 
\end{proof}

We establish a Littlewood--Paley-type inequality associated with the height projections. We need it to prove the endpoint $L^4$ estimates. 

\begin{lemma}[Littlewood--Paley type inequality]\label{prop:LP}
Let $p\in [2,\infty)$ and let $\vomega\in \R^{\nu}$ be non-resonant. For any trigonometric polynomial $f\in \mathcal{T}_{\vec{\omega}}(\R)$, we have
\[
  \left(\limsup_{L\to\infty}\frac{1}{2L}\int_{-L}^{L}\bigl|f(x)\bigr|^p\,dx\right)^{\frac{1}{p}}
  \lesssim
  \left(\sum_{N\in 2^{\N}}
    \left(\limsup_{L\to\infty}\frac{1}{2L}\int_{-L}^{L}\bigl|R_{N}f(x)\bigr|^p\,dx\right)^{\frac{2}{p}}
  \right)^{\frac 12}. 
\]
\label{lem:littlewood_paley}
\end{lemma}

\begin{proof}
Since $f$ is a trigonometric polynomial, we can apply Lemma~\ref{lemma:kronecker_weyl}
and obtain
\[\|f\|_{\mathcal{L}^{p}(\R)} = \|f\|_{L^{p}(\T^\nu)}. \]
The Littlewood--Paley inequality on the torus reveals that 
\[\|f\|_{L^{p}(\T^\nu)} \lesssim 
\Bigl(\sum_{N\in 2^{\N}}\|P_{N}f\|^2_{L^{p}(\T^\nu)}\Bigr)^{\frac 12}. \]
where $P_{N}f(x) = \sum_{|k|\sim N}e^{2\pi i k\cdot x}\widehat{f}(k)$. 
Applying Lemma~\ref{lemma:kronecker_weyl} to each summand 
$\|P_{N}f\|^2_{L^{p}(\T^\nu)}$, we conclude our desired inequality. 
\end{proof}

There is also the Littlewood--Paley type inequality for the frequency projections $P_{N}$: 
\begin{lemma}[Proposition 4.1 in \cite{Schippa_quasiperiodic}]
Let $p\in [2,\infty)$ and let $\vomega\in \R^{\nu}$ be non-resonant. For any trigonometric polynomial $f\in \mathcal{T}_{\vec{\omega}}(\R)$, we have
\[
  \left(\limsup_{L\to\infty}\frac{1}{2L}\int_{-L}^{L}\bigl|f(x)\bigr|^p\,dx\right)^{\frac{1}{p}}
  \lesssim
  \left(\sum_{N\in 2^{\N}}
    \left(\limsup_{L\to\infty}\frac{1}{2L}\int_{-L}^{L}\bigl|P_{N}f(x)\bigr|^p\,dx\right)^{\frac{2}{p}}
  \right)^{\frac 12}. 
\]
\label{lem:LP_Schippa}
\end{lemma}

There are the Bernstein-type inequalities for quasi-periodic
functions: 
\begin{lemma}
  Let $\nu \ge 2$, $\vec{\omega}\in\R^{\nu}$ be non-resonant, 
  $p, q \in [2,\infty)$ with $p \le q$ and $C > 1$. Then, the following inequality
  \begin{equation}
    \|R_C f\|_{\mathcal{L}^q (\R)} \lesssim 
    C^{\nu\bigl(\frac{1}{p} - \frac{1}{q}\bigr)} 
    \|R_C f\|_{\mathcal{L}^p (\R)}
  \end{equation}
  holds for any $f\in \mathcal{T}_{\vec{\omega}}(\R)$. Moreover, let $P_{I}$ be a 
  frequency projection onto an interval $I$ and $X$ denote the number of lattice 
  points in 
  \[\{k\in \Z^\nu \; ; \; |k| \leq C,\; \vomega \cdot k \in I\}. \]
  Then we have
  \begin{equation}
    \|P_{I}R_C f\|_{\mathcal{L}^q (\R)} \lesssim 
    X^{\bigl(\frac{1}{p} - \frac{1}{q}\bigr)} 
    \|P_{I}R_C f\|_{\mathcal{L}^p (\R)}
  \end{equation}
  for all $f\in \mathcal{T}_{\vec{\omega}}(\R)$. 
  \label{lemma:Bernstein}
\end{lemma}

\begin{proof}
  Combining Lemma~\ref{lemma:kronecker_weyl} and the Bernstein inequality on the torus, 
  we have
  \begin{align*}
    \|R_C f\|_{\mathcal{L}^q (\R)} &= \|h_C\|_{L^q (\T^\nu)}\\
    &\lesssim C^{\nu \bigl(\frac{1}{p} - \frac{1}{q}\bigr)}\|h_C\|_{L^p (\T^\nu)}\\
    &= C^{\nu \bigl(\frac{1}{p} - \frac{1}{q}\bigr)}\|R_C f\|_{\mathcal{L}^p (\R)}. 
  \end{align*}
  This completes the proof of the first inequality. The left hand side of the 
  second inequality is 
  \[\Bigl(\limitsup{L}{\infty}\frac{1}{2L}\int^{L}_{-L}
  \Bigl|\sum_{\substack{k\in \Z^\nu, |k|\leq C 
  \\\vec{\omega}\cdot k \in I}}\widehat{f}(\vec{\omega}\cdot k)
  e^{-2\pi i (\vec{\omega}\cdot k)x}\Bigr|^{q}dx\Bigr)^{\frac 1q}\]
  From Lemma~\ref{lemma:kronecker_weyl} and 
  the Bernstein on torus, it holds that 
  \[\|P_{I}R_C f\|_{\mathcal{L}^q (\R)}
  \lesssim
  X^{\bigl(\frac{1}{p} - \frac{1}{q}\bigr)} \|P_{I}R_C f\|_{\mathcal{L}^p (\R)}. \]
\end{proof}

\begin{lemma}
  Let $\nu \geq 2$, $\vec{\omega}\in\R^\nu$ be non-resonant
  and $\eta\in\mathcal{S}(\R)$ be a Schwartz function that satisfies
  \begin{align*}
    \begin{cases}
      \supp \mathcal{F}_{\R}[\eta]\subset [-1,1],\\
      \eta \gtrsim 1\; \mathrm{on\;}[-1,1]. 
    \end{cases}
  \end{align*}
  Then for all 
   trigonometric polynomials $f\in\mathcal{T}_{\vec{\omega}}(\R)$,
  we have
  \[\limitsup{L}{\infty}\frac{1}{2L}\int_{\R}|f(x)|^4 \eta(x / L)dx
  \sim \limitsup{L}{\infty}\frac{1}{2L}\int^{L}_{-L}|f(x)|^{4}dx\]
  \label{lemma:weighted_lp}
\end{lemma}

\begin{proof}
  Let $\Lambda\subset\Z^{\nu}$ be a finite subset and let $f(x) = \sum_{k\in\Lambda}e^{2\pi i(k\cdot\vomega)x}a_{k}$. Noting that 
  \[
  |f(x)|^4 = \Bigl|\sum_{k\in\Lambda}e^{2\pi i(k\cdot\vomega)}a_{k}\Bigr|^{4} = \sum_{k_1,k_2, k_3, k_4}a_{k_1}\overline{a_{k_2}}a_{k_3}\overline{a_{k_4}}e^{2\pi i (k_1 - k_2 + k_3 - k_4)\cdot\vomega}, 
  \]
  the integral becomes
  \begin{align*}
    \limsup_{L\to\infty}\frac{1}{2L}\int_{\R}|f(x)|^4 \eta(x/L)dx
    &= \limsup_{L\to\infty}\frac{1}{2L}\int_{\R}\sum_{k_1,k_2, k_3, k_4}a_{k_1}\overline{a_{k_2}}a_{k_3}\overline{a_{k_4}}e^{2\pi i (k_1 - k_2 + k_3 - k_4)\cdot\vomega}\eta(x/L)dx\\
    &=\sum_{k_1 + k_3 = k_2 + k_4}\limsup_{L\to\infty}\frac{1}{2L}\int_{\R}a_{k_1}\overline{a_{k_2}}a_{k_3}\overline{a_{k_4}}\eta(x/L)dx\\
    &+\sum_{k_1 + k_3 \neq k_2 + k_4}\limsup_{L\to\infty}\frac{1}{2L}\int_{\R}a_{k_1}\overline{a_{k_2}}a_{k_3}\overline{a_{k_4}}e^{2\pi i (k_1 - k_2 + k_3 - k_4)\cdot\vomega}\eta(x/L)dx\\
    &= \frac{1}{2}\widehat{\eta}(0)\sum_{k_1 + k_3 = k_2 + k_4}a_{k_1}\overline{a_{k_2}}a_{k_3}\overline{a_{k_4}}\\
    &\quad+ \frac{1}{2}\sum_{k_1 + k_3 \neq k_2 + k_4}a_{k_1}\overline{a_{k_2}}a_{k_3}\overline{a_{k_4}}\limsup_{L\to\infty}\widehat{\eta}(L(k_1 - k_2 + k_3 - k_4)\cdot\vomega). 
  \end{align*}
  Since $\eta$ is smooth, $\limsup_{L\to\infty}\widehat{\eta}(L(k_1 - k_2 + k_3 - k_4)\cdot\vomega) = 0$. Thus, 
  \[
  \limsup_{L\to\infty}\frac{1}{2L}\int_{\R}|f(x)|^4 \eta(x/L)dx
  = \frac{1}{2}\widehat{\eta}(0)\sum_{k_1 + k_3 = k_2 + k_4}a_{k_1}\overline{a_{k_2}}a_{k_3}\overline{a_{k_4}} = \frac{1}{2}\widehat{\eta}(0)\|f\|^{4}_{\mathcal{L}^4 (\R)}. 
  \]
  This completes the proof. 
\end{proof}

\section{Examples}
In this section, we present some examples illustrating the necessary conditions for our Strichartz estimates. We consider the lower bound for 
the quantity $\|e^{2\pi itD^{2}_{x}} R_{C}f
  \|_{L^p_{t}\mathcal{L}^p_{x}([0,T]\times \R)} / \|R_{C}f\|_{\mathcal{L}^2_{x}(\R)}$
  for each example. Throughout this section, $\nu \geq 2$ and $\vomega\in\R^\nu$
  is non-resonant.

\begin{enumerate}[label = (\alph*)]
  \item Let $f(x) = e^{2\pi i(k\cdot\vomega)x}$ with $|k|\sim C$. Then, we immediately find that 
  \[\|e^{2\pi itD^{2}_{x}} R_{C}f
  \|_{L^p_{t}\mathcal{L}^p_{x}([0,T]\times \R)}
  / \|R_{C}f\|_{\mathcal{L}^2_{x}(\R)}
  \gtrsim T^{\frac 1p}. \]
  From this example, we notice that our estimate \eqref{eq:main_thm_degree2}
  with $p = 6$ and $T\sim C^2$ is sharp up to some subpolynomial factor. Also 
  the estimate \eqref{eq:main_thm_degree3} with $p = 6$ and $T \sim C^{4}$
  is sharp. 

  \item Let $f(x) = \sum_{|k|\leq C}e^{2\pi i (k\cdot\omega)x}$. Lemma~\ref{lemma:kronecker_weyl}
  reveals that 
  \[
    \int^{T}_{0}\limsup_{L\to\infty}\frac{1}{2L}
    \int^{L}_{-L}\Bigl|\sum_{|k|\leq C}
    e^{2\pi i[(k\cdot\vomega)x + (k\cdot\vomega)^2 t] }\Bigr|^p dxdt
    = \int^{T}_{0}\int_{\T^\nu}\Bigl|\sum_{|k|\leq C}
    e^{2\pi i[k\cdot x + (k\cdot \vomega)^2 t]}\Bigr|^p dxdt
  \]
  and
  \[\limsup_{L\to\infty}\frac{1}{2L}\int^{L}_{-L}
  \Bigl|\sum_{|k|\leq C}e^{2\pi i(k\cdot\vomega)x}\Bigr|^2dx 
  = \int_{\T^\nu}\Bigl|\sum_{|k|\leq C}e^{2\pi ik\cdot y}\Bigr|^2dy. \]
  Therefore, taking $T \leq C^{-2}$, our quantity satisfies 
  \[\|e^{2\pi itD^{2}_{x}} R_{C}f
  \|_{L^p_{t}\mathcal{L}^p_{x}([0,T]\times \R)}
  / \|R_{C}f\|_{\mathcal{L}^2_{x}(\R)}
  \gtrsim C^{\frac{\nu}{2} - \frac{\nu + 2}{p}}. \]
  This example shows that Schippa's result (Theorem~\ref{theorem:schippa_decoupling})
  with $d = 1$ and $p\geq 6$ is sharp up to some subpolynomial factor. 

  \item This example was inspired by \cite{Degn_Germain_Guth_Myerson}. We present this example when $\nu = 2$ and $\omega\in\R\backslash\Q$ is not of degree $2$. This example makes our heuristics rigorous when $p\geq 6$ is an even integer. The key ingredient here is a result for the Parsell--Vinogradov systems. For integers $s\geq 1$, and $N\geq 1$, we consider the following system: 
  \begin{align*}
    X_{1} + \cdots + X_{s} &= X_{s + 1} + \cdots + X_{2s},\\
    Y_{1} + \cdots + Y_{s} &= Y_{s + 1} + \cdots + Y_{2s},\\
    X^2_{1} + \cdots + X^2_{s} &= X^2_{s + 1} + \cdots + X^2_{2s},\\
    X_{1}Y_{1} + \cdots + X_{s}Y_{s} &= X_{s + 1}Y_{s + 1} + \cdots +X_{2s}Y_{2s}\\
    Y^2_{1} + \cdots + Y^2_{s} &= Y^2_{s + 1} + \cdots + Y^2_{2s},
  \end{align*}
  with $-N \leq X_{i}, Y_{i}\leq N$. Let $J_{s}(N)$ denote the number of integer solutions to this system. In \cite{Persell_Prendiville_Wooley}, the following lower bound for $J_{s}(N)$ was established: 
  \begin{theorem}[Theorem 1.2. in \cite{Persell_Prendiville_Wooley}]
    Suppose $s\geq 1$ is a natural number and $N \geq 1$. Then one has
    \begin{equation}
      J_{s}(N) \gtrsim N^{2s} + N^{2s - 5} + N^{4s - 8}. 
      \label{eq:PV_lower}
    \end{equation}
  \end{theorem}
  Now, we are ready to construct our example. Let $\omega\in\R\backslash\Q$ be not of degree $2$ and let $f(x) = \sum_{|k|\leq C}e^{2\pi i(k\cdot\vomega)x}$. Let $\chi$ be a smooth function supported on $[-1,1]$ with non-negative Fourier transform (such a function can be obtained by the convolution $\varphi * \tilde{\varphi}$ where $\varphi\in C^{\infty}_{0}([-1/2,1/2])$ and $\tilde{\varphi}(x):= \varphi(-x)$). Then for even $p\geq 6$, multiplying out the $L^{p}_{t}\mathcal{L}^{p}_{x}$-norm, it follows that 
  \begin{align*}
    &\quad\Bigl\|\sum_{|k|\leq C}e^{2\pi i [(k\cdot \vomega)x + (k\cdot\vomega)^2 t]}\Bigr\|^{p}_{L^{p}_{t}\mathcal{L}^{p}_{x}([0,T]\times \R)}\\
    &= \frac{1}{2}\Bigl\|\sum_{|k|\leq C}e^{2\pi i [(k\cdot \vomega)x + (k\cdot\vomega)^2 t]}\Bigr\|^{p}_{L^{p}_{t}\mathcal{L}^{p}_{x}([-T,T]\times \R)}\\
    &\gtrsim \int_{\R}\limsup_{L\to\infty}\frac{1}{2L}\int^{L}_{-L}\Bigl|\sum_{|k|\leq C}e^{2\pi i [(k\cdot \vomega)x + (k\cdot\vomega)^2 t]}\Bigr|^{p}dx\chi\Bigl(\frac{t}{T}\Bigr)dt\\
    &=\sum_{k_{1} + \cdots + k_{p/2} = k_{p/2 + 1}+\cdots + k_{p}}
    \int_{\R}e^{2\pi i [\{(k_1\cdot\vomega)^2 + \cdots +(k_{p/2}\cdot\vomega)^2\} - \{(k_{p/2 + 1}\cdot\vomega)^2 + \cdots +(k_{p}\cdot\vomega)^2\}]t}\chi\Bigl(\frac{t}{T}\Bigr)dt. 
  \end{align*}
  Hereafter, we abbreviate the condition $k_{1} + \cdots + k_{p/2} = k_{p/2 + 1}+\cdots + k_{p}$ to $(*)$ and $\{(k_1\cdot\vomega)^2 + \cdots +(k_{p/2}\cdot\vomega)^2\} - \{(k_{p/2 + 1}\cdot\vomega)^2 + \cdots +(k_{p}\cdot\vomega)^2\}$ to $\Omega$. Then 
  \begin{align*}
    \sum_{(*)}\int_{\R}e^{2\pi i\Omega t}\chi\Bigl(\frac{t}{T}\Bigr)dt
    &= \sum_{(*)} T\widehat{\chi}(-T\Omega)\\
    &= \sum_{(*),\Omega = 0} T\widehat{\chi}(0)+\sum_{(*),\Omega \neq 0} T\widehat{\chi}(-T\Omega). 
  \end{align*}
  Note that the condition $(*)$ with $\Omega = 0$ forms a Parsell--Vinogradov system of $(k_{1},\cdots , k_{p})$ because $1$, $\omega$, and $\omega^2$ are linearly independent over $\Q$. 
  Since $\widehat{\chi}$ is non-negative, applying \eqref{eq:PV_lower} with $s = p/2$, we obtain
  \[
  \Bigl\|\sum_{|k|\leq C}e^{2\pi i [(k\cdot \vomega)x + (k\cdot\vomega)^2 t]}\Bigr\|^{p}_{L^{p}_{t}\mathcal{L}^{p}_{x}([0,T]\times \R)} / \|f\|_{\mathcal{L}^{2}_{x}}
  \gtrsim
  \Bigl(\sum_{(*),\Omega = 0} T\Bigr)^{\frac 1p} C^{-1}
  \gtrsim
  C^{1 - \frac 8p}T^{\frac 1p}. 
  \]
  Note that this example becomes sharper as $T$ gets larger because the second term $\sum_{(*),\Omega \neq 0} T\widehat{\chi}(-T\Omega)$ rapidly decays as $T\to\infty$. Also note that the exponents of $C$ and $T$ coincide with the second term of Observation 2.
\end{enumerate}
\section{Proof of Theorem \ref{theorem:long_time_Strchartz}} 
We first prepare a counting lemma that will be repeatedly used in this and the next section. 
\begin{lemma}
    Let $\varepsilon > 0$, $C > 0$, and $\omega\in \R$ be an irrational real number. 
    For fixed $M, S \in \R$, consider 
    \begin{align}
        \begin{cases}
            \lambda_1 + \lambda_2 + \lambda_3 = M\\
            \lambda^2_1 + \lambda^2_2 + \lambda^2_3 = S
        \end{cases}
        \label{eq:additive_energy6}
    \end{align}
    where $\lambda_1, \lambda_2, \lambda_3 \in \Z + \omega \Z$. Then, we have the following bound
    \[\#\{(k_1, k_2, n_1,n_2, m_1,m_2)\in \Z^6\cap [-C, C]^{6} \;;\; \lambda_1 = k_1 + k_2\omega, \lambda_2 = n_1 + n_2 \omega, \lambda_3 = m_1 + m_2\omega \;\mathrm{with}\; \eqref{eq:additive_energy6} \} \lesssim_{\varepsilon} C^{\varepsilon}. \]\label{lemma:additive_energy}
\end{lemma}
\begin{proof}
    Since $(k_1, k_2, n_1,n_2, m_1,m_2)\in \Z^6\cap [-C, C]^{6}$, 
    we may assume $M \lesssim C$ and $S\lesssim C^2$. Furthermore, we may
    assume $M\in \Z + \omega \Z$ and $S\in \Z + \omega \Z + \omega^2\Z$ otherwise
    there is no $(\lambda_1, \lambda_2, \lambda_3)\in \Z + \omega\Z$ 
    satisfying \eqref{eq:additive_energy6}. 
    
    \smallskip
    \noindent\textbf{Case 1.}
    We first consider the case where the degree of the number $\omega$
    is greater than $2$. From the first equation in \eqref{eq:additive_energy6}, 
    we eliminate $\lambda_3$ from the second equation and get
    \begin{equation}
        (3\lambda_1 - M)^2 + (3\lambda_2 - M)^2 + (3\lambda_1 - M)(3\lambda_2 - M) = \frac{9S - 3M^2}{2}.
        \label{eq:reduced_additive_energy}
    \end{equation}
    Set $L := 9S - 3M^2$, $K_{1}+ N_1 \omega := 3\lambda_1 - M$, and $K_2 + N_2 \omega := 3\lambda_2 - M$ ($K_1,K_2,N_1,N_2\in \Z$). 
    If $K_1,K_2,N_1,N_2$ are determined, then $k_1, k_2, n_1, n_2$ are also determined once $M$ is fixed. 
    Therefore, it suffices to count the number of integers $K_1,K_2,N_1, N_2$ with \eqref{eq:reduced_additive_energy}. Observe that 
    \begin{align*}
        (K_1 + N_1\omega)^2 &= K^2_1 + 2K_1N_1\omega + N^2_1\omega^2\\
        (K_2 + N_2\omega)^2 &= K^2_2 + 2K_2N_2\omega + N^2_2\omega^2\\
        (K_1 + N_1\omega)(K_2 + N_2\omega) &= K_1K_2 + (K_1N_2 + K_2N_1)\omega + N_1N_2\omega^2
    \end{align*}
    and $1, \omega, \omega^2$ are linearly independent over $\Q$ 
    because the degree of $\omega$ is greater than $2$ now. Hence, we are led 
    to consider the following system:
    \begin{align}
        \begin{cases}
            2(K^2_1 + K^2_2 + K_1K_2) &= L_1\\
            2(2K_1N_1 + 2K_2N_2 + K_1N_2 + K_2N_1) &= L_2\\
            2(N^2_1 + N^2_2 + N_1N_2) &= L_3
        \end{cases}
        \label{eq:degree3_counting}
    \end{align}
    where $L_1, L_2, L_3\in \Z$. 
    The left hand side of the first equation in \eqref{eq:degree3_counting}
    can be decomposed in the ring of Eisenstein integers $\Z + \frac{1 + \sqrt{-3}}{2}\Z$ as
    \[2\Bigl(K_1 + \frac{1 - \sqrt{-3}}{2}K_2\Bigr)\Bigl(K_1 + \frac{1 + \sqrt{-3}}{2}K_2\Bigr). \]
    Note that the ring of Eisenstein integers is a unique factorization domain. Then, it follows that 
    the number of divisors of $L_1$ in this ring is $\lesssim_{\varepsilon} L^{\varepsilon}_1$ and hence the number of 
    pairs $(K_1, K_2)$ satisfying the first equation in \eqref{eq:degree3_counting} is also $\lesssim_{\varepsilon} L^{\varepsilon}_1$. 
    We apply the same argument to the third equation in \eqref{eq:degree3_counting} and obtain that 
    the number of possible pairs $(N_1, N_2)$ is $\lesssim_{\varepsilon}L^{\varepsilon}_{3}$. Thus, the number of
    possible $(K_1, K_2, N_1, N_2)$ is $\lesssim_{\varepsilon} L^{\varepsilon}_1L^{\varepsilon}_3 \lesssim C^{2\varepsilon}$.
    Replacing $\varepsilon$ by $\varepsilon / 2$ yields our desired bound\footnote{This bound also can be proved by using Theorem~\ref{thm:BP_counting}. }. 

    \smallskip
    \noindent\textbf{Case 2.}
    Next, we consider when the degree of $\omega$ is $2$. The argument is the
    same as in the Case 1 up to \eqref{eq:degree3_counting}. However, 
    in this case, $1, \omega, \omega^2$ are not linearly independent over $\Q$. 
    So, assume that $\omega$ is a solution to a quadratic equation
    \[a x^2 + bx + c = 0\]
    where $a,b,c\in\Z$, $a>  0$, and $b^2 - 4ac > 0$ (we can assume this discriminant condition
    because $\omega$ is a real number). We divide the argument into two cases according to whether $b =0 $. 
    
    \smallskip
    \noindent\textbf{Case 2-1 ($\mathbf{b = 0}$).} In this case, $\omega^2 = -\frac{c}{a}$. 
    Hence, instead of \eqref{eq:degree3_counting}, we consider the following system: 
    \begin{align}
        \begin{cases}
            2\{a(K^2_1 + K_1K_2 + K^2_2) - c(N^2_1 + N_1N_2 + N^2_2)\} = L_1\\
            2\{2K_1N_1 + 2K_2N_2 + K_1N_2 + K_2N_1\} = L_2
        \end{cases}
        \label{eq:degree2_counting_b0}
    \end{align}
    where $L_1, L_2\in \Z$. Set $K = (K_1,K_2)$, $N = (N_1, N_2)$, and $Q(X) := X^2_1 + X_1X_2 + X^2_2$
    for $X = (X_1,X_2)\in \R^2$. From a simple computation, it follows that
    \[4Q(K)Q(N) - \frac{L^2_2}{4} = 3 (K_1N_2 - K_2N_1)^2. \]
    We write $\Delta(K,N) = K_1N_2 - K_2N_1$. Since $aQ(K) = cQ(N) + \frac{L_1}{2}$ , this 
    equation becomes
    \[a\Bigl(4cQ(N) + L_1 \Bigr)^2 - 12ac (\Delta(K,N))^2 =L^2_1 + ac L^2_2. \]
    Since $-ac > 0$, $X^2 -12acY^2 = R$ is a real analytic image of $\mathbb{S}^1$. Applying Theorem~\ref{thm:BP_counting}, we obtain 
    that the number of possible values of $(Q(N), \Delta(K,N))$ is 
    $\lesssim_{\varepsilon} C^{\varepsilon}$
    because $L_1, L_2$ are fixed and also $|L_1|, |L_2|\lesssim C^{2}$. 
    For each $Q(N) = P$, there are $\lesssim_{\varepsilon}P^{\varepsilon}$ choices of $(N_1,N_2)$ as explained in the Case 1. If $(N_1, N_2) \neq (0,0)$ are given, then combining the second equation of \eqref{eq:degree2_counting_b0} with $\Delta(K,N) = P'$, 
    $(K_1, K_2)$ are determined. If $N_1 = N_2 = 0$, we can bound the number of pairs $(K_1,K_2)$ by $\lesssim_{\varepsilon}C^{\varepsilon}$ from the first equation of \eqref{eq:degree2_counting_b0}. This shows that the number of the
    possible choices of $(K_1,K_2, N_1, N_2)$ is $\lesssim_{\varepsilon}C^{\varepsilon}$ and this gives our desired bound. 

    \smallskip
    \noindent\textbf{Case 2-2 ($\mathbf{b \neq 0}$).} In this case, $\omega^2$ can be written as
    \[\omega^2 = -\frac ba \omega - \frac ca.\]
    Observe that 
    \begin{align*}
        (K_1 + N_1\omega)^2 &= \Bigl(K^2_1 - \frac{c}{a}N^2_1\Bigr) + \Bigl(2K_1N_1 - \frac{b}{a}N^2_1\Bigr)\omega,\\
        (K_2 + N_2\omega)^2 &= \Bigl(K^2_2 - \frac{c}{a}N^2_2\Bigr) + \Bigl(2K_2N_2 - \frac{b}{a}N^2_2\Bigr)\omega,\\
        (K_1 + N_1\omega)(K_2 + N_2\omega) &= \Bigl(K_1K_2 - \frac{c}{a}N_1N_2\Bigr)+ \Bigl(K_1N_2 + K_2N_1 - \frac{b}{a}N_1N_2\Bigr)\omega. 
    \end{align*}
    Then, we are led to consider the following system: 
    \begin{align}
        \begin{cases}
            2\{a(K^2_1 + K_1K_2 + K^2_2) - c(N^2_1 + N_1N_2 + N^2_2)\} = L_1\\
            2\{a(2K_1 N_1 + 2K_2 N_2 + K_1 N_2 + K_2 N_1) - b(N^2_1 + N_1N_2 + N^2_2)\} = L_2
        \end{cases}
    \end{align}
    for fixed $L_1, L_2\in \Z$. Similar to the Case 2-1., computing $4Q(K)Q(N) - \frac{1}{4a^2}(L_2 + 2bQ(N))^2$, we find that
    \[\{2DQ(N) - (2aL_1 - bL_2)\}^2 + 3D(2a\Delta(K,N))^2 = D(2aL_1 -bL_2)^2- DL^2_2\]
    where $D := b^2 - 4ac > 0$. Applying Theorem~\ref{thm:BP_counting} again and arguing as in the Case 2-1, it follows that the number of possible tuples $(K_1,K_2, N_1, N_2)$ is $\lesssim_{\varepsilon}C^{\varepsilon}$. 
\end{proof}

Now we are ready to prove Theorem \ref{theorem:long_time_Strchartz}. 
We find that it suffices to show
\begin{equation}
  \|e^{2\pi i tD^{2}_{x}}R_{C}f(x)\|_{L^{6}_{t}
\mathcal{L}^{6}_{x}([0,C^{2}]\times\R)}\lesssim_{\varepsilon} 
C^{\frac 13 + \varepsilon}\|R_{C}f(x)\|_{\mathcal{L}^2_{x}(\R)}
\label{eq:L6_estimate}
\end{equation}
because the general case follows by interpolating this with the trivial case $p =\infty$ and dividing the time interval $[0,T]$ into 
subintervals with length $C^2$. 
Let $\Lambda_{\omega, C} 
= \{\lambda = k_1 + \omega k_2\in \Z + \omega\Z\; ; \;|(k_1, k_2)|\sim C\}$. We expand the left-hand side of 
\eqref{eq:L6_estimate} as
\begin{align*}
  &\|e^{2\pi i tD^{2}_{x}}R_{C}f(x)\|^6_{L^{6}_{t}
  \mathcal{L}^{6}_{x}([0,C^{2}]\times\R)} \\
  &= \int^{C^2}_{0}
  \limitsup{L}{\infty}\frac{1}{2L}\int^{L}_{-L}|e^{2\pi itD^{2}_{x}}f(x)|^6
  dxdt\\
  &= \int^{C^2}_{0}
  \limitsup{L}{\infty}\frac{1}{2L}\int^{L}_{-L}\Bigl|\sum_{\lambda\in \Lambda_{\omega, C}}e^{2\pi i[\lambda x + |\lambda|^2 t]}a_{\lambda}\Bigr|^6 dxdt\\
  &= \int^{C^2}_{0}\sum_{\lambda_1 + \lambda_3 + \lambda_5= \lambda_2 + \lambda_4 + \lambda_6}
  e^{2\pi i\{(|\lambda_1|^2 + |\lambda_3|^2 + |\lambda_5|^2) - (|\lambda_2|^2 + |\lambda_4|^2 + |\lambda_6|^2)\}t}
  a_{\lambda_1}a_{\lambda_3}a_{\lambda_5} 
  \overline{a_{\lambda_2} a_{\lambda_4} a_{\lambda_6}}dt
\end{align*}
where we applied
\begin{align*}
  \limitsup{L}{\infty}\frac{1}{2L}\int^{L}_{-L}e^{2\pi i\{(\lambda_1 + \lambda_3 + \lambda_5) - (\lambda_2 + \lambda_4 + \lambda_6)\}x} 
  =\begin{cases}
    1 \quad (\lambda_1 + \lambda_3 + \lambda_5 = \lambda_2 + \lambda_4 + \lambda_6)\\
    0 \quad (\lambda_1 + \lambda_3 + \lambda_5 \neq \lambda_2 + \lambda_4 + \lambda_6).
  \end{cases}
\end{align*} 
\noindent{\textbf{Step 1. }}We first show \eqref{eq:L6_estimate} when $\bigl|(|\lambda_1|^2 + |\lambda_3|^2 + |\lambda_5|^2) 
- (|\lambda_2|^2 + |\lambda_4|^2 + |\lambda_6|^2)\bigr|< C^{-2}$. In this case,
we bound the above integral as
\begin{align*}
  &\Bigl|\int^{C^2}_{0}\sum_{\substack{\lambda_1 + \lambda_3 + \lambda_5 = \lambda_2 + \lambda_4 + \lambda_6\\
  \bigl|(|\lambda_1|^2 + |\lambda_3|^2 + |\lambda_5|^2) 
- (|\lambda_2|^2 + |\lambda_4|^2 + |\lambda_6|^2)\bigr|< C^{-2}}}
  e^{2\pi i\{(|\lambda_1|^2 + |\lambda_3|^2 + |\lambda_5|^2) - (|\lambda_2|^2 + |\lambda_4|^2 + |\lambda_6|^2)\}t}
  a_{\lambda_1}a_{\lambda_3}a_{\lambda_5} 
  \overline{a_{\lambda_2} a_{\lambda_4} a_{\lambda_6}}dt\Bigr|\\
  &\leq C^{2}\sum_{\substack{\lambda_1 + \lambda_3 + \lambda_5 = \lambda_2 + \lambda_4 + \lambda_6\\
  \bigl|(|\lambda_1|^2 + |\lambda_3|^2 + |\lambda_5|^2) 
- (|\lambda_2|^2 + |\lambda_4|^2 + |\lambda_6|^2)\bigr|< C^{-2}}}
  \bigl|a_{\lambda_1}a_{\lambda_3}a_{\lambda_5} 
  \overline{a_{\lambda_2} a_{\lambda_4} a_{\lambda_6}}\bigr|
\end{align*}
Hence, it is enough to show 
\[\sum_{\substack{\lambda_1 + \lambda_3 + \lambda_5 = \lambda_2 + \lambda_4 + \lambda_6\\
  \bigl|(|\lambda_1|^2 + |\lambda_3|^2 + |\lambda_5|^2) 
- (|\lambda_2|^2 + |\lambda_4|^2 + |\lambda_6|^2)\bigr|< C^{-2}}}
  \bigl|a_{\lambda_1}a_{\lambda_3}a_{\lambda_5} 
  \overline{a_{\lambda_2} a_{\lambda_4} a_{\lambda_6}}\bigr|\lesssim_{\varepsilon}C^{\varepsilon}.\]
From the AM-GM inequality, it follows that
\begin{align*}
  &\sum_{\substack{\lambda_1 + \lambda_3 + \lambda_5 = \lambda_2 + \lambda_4 + \lambda_6\\
  \bigl|(|\lambda_1|^2 + |\lambda_3|^2 + |\lambda_5|^2) 
- (|\lambda_2|^2 + |\lambda_4|^2 + |\lambda_6|^2)\bigr|< C^{-2}}}
  \bigl|a_{\lambda_1}a_{\lambda_3}a_{\lambda_5} 
  \overline{a_{\lambda_2} a_{\lambda_4} a_{\lambda_6}}\bigr|\\
  &\lesssim \sum_{\substack{\lambda_1 + \lambda_3 + \lambda_5 = \lambda_2 + \lambda_4 + \lambda_6\\
  \bigl|(|\lambda_1|^2 + |\lambda_3|^2 + |\lambda_5|^2) 
- (|\lambda_2|^2 + |\lambda_4|^2 + |\lambda_6|^2)\bigr|< C^{-2}}}
|a_{\lambda_1}a_{\lambda_3}a_{\lambda_5}|^2 \\
&\quad+  \sum_{\substack{\lambda_1 + \lambda_3 + \lambda_5 = \lambda_2 + \lambda_4 + \lambda_6\\
  \bigl|(|\lambda_1|^2 + |\lambda_3|^2 + |\lambda_5|^2) 
- (|\lambda_2|^2 + |\lambda_4|^2 + |\lambda_6|^2)\bigr|< C^{-2}}}
|a_{\lambda_2}a_{\lambda_4}a_{\lambda_6}|^2. 
\end{align*}
By symmetry and the Cauchy--Schwarz inequality, 
it is enough to show the following bound
\[\Bigl(\sum_{\lambda} |a_{\lambda}|^2\Bigr)^3 
\sup_{\lambda_1, \lambda_3, \lambda_5}\#\Bigl\{ (\lambda_2, \lambda_4, \lambda_6)\; ;\; \substack{\lambda_1 + \lambda_3 + \lambda_5 = \lambda_2 + \lambda_4 + \lambda_6\\
  |(|\lambda_1|^2 + |\lambda_3|^2 + |\lambda_5|^2) 
- (|\lambda_2|^2 + |\lambda_4|^2 + |\lambda_6|^2)|< C^{-2}} \Bigr\}
\lesssim_{\varepsilon}C^{\varepsilon}\Bigl(\sum_{\lambda} |a_{\lambda}|^2\Bigr)^3. \]
Thus, the remaining task is to establish the bound
\[\#\Bigl\{(\lambda_2, \lambda_4, \lambda_6)\; ;\; \substack{\lambda_2 + \lambda_4 + \lambda_6 = M\\
  \bigl||\lambda_2|^2 + |\lambda_4|^2 + |\lambda_6|^2 - S\bigr|< C^{-2}}\Bigr\}\lesssim_{\varepsilon}C^{\varepsilon}\]
for fixed $M, S$ with $|M|\lesssim C$ and $|S| \lesssim C^2$. 
Hereafter, we assume that 
$\omega$ is a solution to a quadratic equation
\[ax^2 + bx + c = 0\]
where $a,b,c\in \Z$ with $a > 0$ and consider 
two subcases. 

\smallskip
\noindent \textbf{Case 1 ($\mathbf{b=0}$). }
In this case, $\omega^2 = -\frac{c}{a}$ and the range of 
$\lambda^2_2 + \lambda^2_4 + \lambda^2_6$ is contained in
$\frac{1}{a}\Z + \omega \Z$, hence the above
set of $(\lambda_2, \lambda_4, \lambda_6)$ can be written as
\[\#\Bigl\{(\lambda_2, \lambda_4, \lambda_6)\; ;\; \substack{\lambda_2 + \lambda_4 + \lambda_6 = M\\
  \bigl||\lambda_2|^2 + |\lambda_4|^2 + |\lambda_6|^2 - S\bigr|< C^{-2}}\Bigr\}=
  \sum_{\substack{|t - S|< C^{-2} \\ t= \frac{1}{a}t_1 + \omega t_2, 
  t_1, t_2\in \Z }}
\# \Bigl\{(\lambda_2,\lambda_4, \lambda_6)\; ; \;
\substack{\lambda_2 + \lambda_4 + \lambda_6 = M\\ 
\lambda^2_2 + \lambda^2_4 + \lambda^2_6 = t}\Bigr\}\]
From Lemma \ref{lemma:additive_energy}, we obtain
\[\# \Bigl\{(\lambda_2,\lambda_4, \lambda_6)\; ; \;
\substack{\lambda_2 + \lambda_4 + \lambda_6 = M\\ 
\lambda^2_2 + \lambda^2_4 + \lambda^2_6 = t}\Bigr\}\lesssim_{\varepsilon} C^{\varepsilon}\]
for each $t_1, t_2 \in \Z$ with $|t - S| < C^{-2}$ 
and $|t_1|, |t_2|\lesssim C^2$. Hence, we have to
count the possible values of $t$, that is, we need to show
\begin{equation}
  \sum_{\substack{|t - S|< C^{-2} \\ t= \frac{1}{a}t_1 + \omega t_2, 
  t_1, t_2\in \Z }} 1
\lesssim_{\varepsilon}C^{\varepsilon}.\label{eq:from_roth}
\end{equation}
Since $S \in \frac{1}{a}\Z + \omega\Z$, $|t - S| < C^{-2}$,
it is enough to show 
\[\# \{m_1, m_2\in \Z\; ; \; |m_1 - \omega m_2| \lesssim C^{-2}, 
\;|m_1|, |m_2| \lesssim C^2\} 
\lesssim_{\varepsilon} C^{\varepsilon}. \]
Since $\omega$ is algebraic, from the Roth theorem (Theorem \ref{theorem:Roth}), we find that 
\[|m_1 - \omega m_2| 
\gtrsim_{\varepsilon} 
\max\{|m_1|^{-1-\varepsilon}, |m_2|^{-1-\varepsilon}\}\] 
and then it follows that
\[|m_1|, |m_2|\gtrsim_{\varepsilon}C^{\frac{2}{1 + \varepsilon}}.\]
Therefore, by combining this bound with $|m_1 - \omega m_2| \lesssim C^{-2}$ 
and $|m_1|, |m_2|\lesssim C^{2}$, we have
\[\# \{m_1, m_2\in \Z\; ; \; |m_1 - \omega m_2| \lesssim C^{-2}, 
\;|m_1|, |m_2| \lesssim C^2\} 
\lesssim_{\varepsilon'} C^{\frac{2\varepsilon'}{1 + \varepsilon'}}
\lesssim_{\varepsilon} C^{\varepsilon}. \]
In the last inequality, we take $\varepsilon = 2\varepsilon'$. From this,
we conclude \eqref{eq:from_roth} and obtain the desired estimate.

\smallskip
\noindent \textbf{Case 2 ($\mathbf{b\neq0}$). }
In this case, $\omega^2 = -\frac{b}{a}\omega - \frac{c}{a}$. Since the
range of $\lambda^2_2 + \lambda^2_4 + \lambda^2_6$ is contained in
$\frac{1}{a}\Z + \frac{\omega}{a}\Z$, we have
\[\#\Bigl\{(\lambda_2, \lambda_4, \lambda_6)\; ;\; \substack{\lambda_2 + \lambda_4 + \lambda_6 = M\\
  \bigl||\lambda_2|^2 + |\lambda_4|^2 + |\lambda_6|^2 - S\bigr|< C^{-2}}\Bigr\}=
  \sum_{\substack{|t - S|< C^{-2} \\ t= \frac{1}{a}t_1 
  + \frac{\omega}{a} t_2, 
  t_1, t_2\in \Z }}
\# \Bigl\{(\lambda_2,\lambda_4, \lambda_6)\; ; \;
\substack{\lambda_2 + \lambda_4 + \lambda_6 = M\\ 
\lambda^2_2 + \lambda^2_4 + \lambda^2_6 = t}\Bigr\}. \]
By an argument similar to that in the Case 1, we reduce our 
estimate to
\[\# \{m_1, m_2\in \Z\; ; \; |m_1 - \omega m_2| \lesssim C^{-2}, 
\;|m_1|, |m_2| \lesssim C^2\} 
\lesssim_{\varepsilon} C^{\varepsilon}\]
and obtain the desired estimate. 
\smallskip

\noindent{\textbf{Step 2. }}For a dyadic number $A \in (C^{-2}, 100C^2]$, define a set 
$\Lambda_{\omega, C,A}$ as
\[\Lambda_{\omega, C,A} := \Bigl\{\lambda_1,\cdots, \lambda_6\in \Lambda_{\omega,C}
\; ; \; \substack{\lambda_1 + \lambda_3 + \lambda_5 = \lambda_2 + \lambda_4 + \lambda_6,\\
A \le |(|\lambda_1|^2 + |\lambda_3|^2 + |\lambda_5|^2) - 
(|\lambda_2|^2 + |\lambda_4|^2 + |\lambda_6|^2)| < 2A}\Bigr\}. \]
We next consider the case when $\lambda_i$'s are contained in these sets.
In this step, from using $\lambda_i\in \Lambda_{\omega,C,A}$, we have the bound
\begin{align*}
  &\int^{C^2}_{0}\sum_{\lambda_i\in \Lambda_{\omega,C,A}}e^{2\pi i\{(|\lambda_1|^2 + |\lambda_3|^2 + |\lambda_5|^2) - (|\lambda_2|^2 + |\lambda_4|^2 + |\lambda_6|^2)\}t}
  a_{\lambda_1}a_{\lambda_3}a_{\lambda_5} 
  \overline{a_{\lambda_2} a_{\lambda_4} a_{\lambda_6}}dt\\
  &\lesssim A^{-1}\sum_{\lambda_i\in \Lambda_{\omega,C,A}}
  |a_{\lambda_1}a_{\lambda_3}a_{\lambda_5}
  a_{\lambda_2}a_{\lambda_4}a_{\lambda_6}|. 
\end{align*}
Hence, it is enough to show 
\[\sum_{\lambda_i\in \Lambda_{\omega,C,A}}
|a_{\lambda_1}a_{\lambda_3}a_{\lambda_5}
a_{\lambda_2}a_{\lambda_4}a_{\lambda_6}|
\lesssim_{\varepsilon} A C^{2+ \varepsilon} 
\Bigl(\sum_{\lambda}|a_{\lambda}|^2\Bigr)^3. \]
We further decompose $\Lambda_{\omega, C,A}$ into the following subsets
\begin{equation}
  \Sigma_{X}:= \{\lambda_1, \cdots, \lambda_6\in 
  \Lambda_{\omega, C,A}\; ;\; 
  \lambda^2_1 + \lambda^2_3 + \lambda^2_5
  -\lambda^2_2 - \lambda^2_4 - \lambda^2_6= X\}\label{eq:dyadic_cases}
\end{equation}
where $A \leq X < 2A$. As in the Step 1, we assume $\omega$ to be
a solution to a quadratic equation 
\[ax^2 + bx + c = 0\]
where $a,b,c\in \Z$ with $a\neq 0$. As we saw in Step 1, depending
on $b$, the
range of 
\[(|\lambda_1|^2 + |\lambda_3|^2 + |\lambda_5|^2) - 
(|\lambda_2|^2 + |\lambda_4|^2 + |\lambda_6|^2)\]
is different. Hence, we set 
\begin{align*}
  R_{b} := \begin{cases}
    \frac{1}{a}\Z + \omega\Z \quad &(b = 0),\\
    \frac{1}{a}\Z + \frac{\omega}{a}\Z &(b\neq 0). 
  \end{cases}
\end{align*}
Then, \eqref{eq:dyadic_cases} is modified as 
\[\sum_{\substack{X\in R_b
\\ A\leq X < 2A}}\sum_{\lambda_i\in \Sigma_{X}}
|a_{\lambda_1}a_{\lambda_3}a_{\lambda_5}
a_{\lambda_2}a_{\lambda_4}a_{\lambda_6}| 
\lesssim_{\varepsilon} A C^{2 + \varepsilon}. \]
Therefore, it suffices to show the following two inequalities:
\begin{itemize}
  \item \begin{equation}
  \sum_{\substack{X\in R_b\\ 
  A \leq X < 2A}}1 \lesssim_{\varepsilon}A C^{2 + \varepsilon}
  \label{eq:Xbound}
  \end{equation}

  \item For each $X\in R_b$ with $A \leq X < 2A$, 
  \begin{equation}
    \sum_{\lambda_i\in \Sigma_{X}}
    |a_{\lambda_1}a_{\lambda_3}a_{\lambda_5}
    a_{\lambda_2}a_{\lambda_4}a_{\lambda_6}| 
    \lesssim_{\varepsilon} C^{\varepsilon}\Bigl(\sum_{\lambda}|a_{\lambda}|^2\Bigr)^3
  \label{eq:Lamx_bound}
\end{equation}
\end{itemize}
We first show \eqref{eq:Xbound}. We only consider the case when $AC^2 > 1$ otherwise 
this estimate is trivial. Since $R_b$ denotes $\frac{1}{a}\Z + \omega\Z$
or $\frac{1}{a}\Z + \frac{\omega}{a}\Z$, it is enough to establish 
\[\#\{m_1 + \omega m_2\; ; \;m_1, m_2\in \Z,\; |m_1|, |m_2|\lesssim C^2,\;
|m_1 + \omega m_2| \sim A\} \lesssim_{\varepsilon} A C^{2 + \varepsilon}. \]
If $A > 1$, then this bound immediately follows because once we fix
$m_1\in \Z$ from $O(C^2)$ choices, then
there are only $O(A)$ choices of $m_2\in \Z$ satisfying $|m_1 + \omega m_2| \sim A$. 
We next consider when $A \leq 1$. From Theorem \ref{theorem:Roth}, it follows that
\[A \gtrsim |m_1 + \omega m_2| \gtrsim_{\varepsilon} |m_1|^{-1-\varepsilon}\]
and hence, we have $|m_1|\gtrsim_{\varepsilon}A^{-\frac{1}{1 + \varepsilon}}$. 
Therefore, there are $O(A^{\frac{1}{1+ \varepsilon}}C^2)$ choices of $m_1$. 
Once $m_1$ is fixed, there are $O(1)$ choices of $m_2$, 
thus we obtain \eqref{eq:Xbound} from bounding $A^{\frac{1}{1 + \varepsilon}}C^2$ 
by $A C^{2 + \varepsilon}$ up to some constants. 

We next show \eqref{eq:Lamx_bound}. The proof is in the same line in the Step 1. 
From the AM-GM inequality, the left hand side is bounded by
\[\lesssim \sum_{\lambda_i \in \Sigma_X}
|a_{\lambda_1}a_{\lambda_3}a_{\lambda_5}|^2 
+  \sum_{\lambda_i \in \Sigma_X}
|a_{\lambda_2}a_{\lambda_4}a_{\lambda_6}|^2. \]
From the symmetry and the Cauchy--Schwarz, it is enough to show
\[\Bigl(\sum_{\lambda} |a_{\lambda}|^2\Bigr)^3 
\sup_{\lambda_1, \lambda_3, \lambda_5}\#\Bigl\{ (\lambda_2, \lambda_4, \lambda_6)\; ;\; \substack{\lambda_1 + \lambda_3 + \lambda_5 = \lambda_2 + \lambda_4 + \lambda_6\\
  |(|\lambda_1|^2 + |\lambda_3|^2 + |\lambda_5|^2) 
- (|\lambda_2|^2 + |\lambda_4|^2 + |\lambda_6|^2)| = X} \Bigr\}
\lesssim_{\varepsilon}C^{\varepsilon}\Bigl(\sum_{\lambda} |a_{\lambda}|^2\Bigr)^3. \]
Hence, the remaining task is to prove the bound
\[\#\Bigl\{(\lambda_2, \lambda_4, \lambda_6)\; ;\; \substack{\lambda_2 + \lambda_4 + \lambda_6 = M\\
  \lambda^2_2 + \lambda^2_4 + \lambda^2_6 = S}\Bigr\}
  \lesssim_{\varepsilon}C^{\varepsilon}\]
for fixed $M, S$ with $|M|\lesssim C$ and $|S| \lesssim C^2$. 
This immediately follows from Lemma \ref{lemma:additive_energy}. 

Now, we have 
\[\sum_{\substack{X\in R_b
\\ A\leq X < 2A}}\sum_{\lambda_i\in \Sigma_{X}}
|a_{\lambda_1}a_{\lambda_3}a_{\lambda_5}
a_{\lambda_2}a_{\lambda_4}a_{\lambda_6}| 
\lesssim_{\varepsilon} A C^{2 + \varepsilon}\Bigl(\sum_{\lambda}|a_{\lambda}|^2\Bigr)^3 \]
and hence 
\[\Bigl|\int^{C^2}_{0}\sum_{\lambda_i\in \Lambda_{\omega, C,A}}
  e^{2\pi i\{(|\lambda_1|^2 + |\lambda_3|^2 + |\lambda_5|^2) - (|\lambda_2|^2 + |\lambda_4|^2 + |\lambda_6|^2)\}t}
  a_{\lambda_1}a_{\lambda_3}a_{\lambda_5} 
  \overline{a_{\lambda_2} a_{\lambda_4} a_{\lambda_6}}dt\Bigr|
  \lesssim C^{2 + \varepsilon}\Bigl(\sum_{\lambda}|a_{\lambda}|^2\Bigr)^3. \]
  This completes the proof. 

\section{Proof of Theorem \ref{theorem:degree3}}
In this section, we assume that $\omega\in \R\backslash\Q$ is an algebraic number that has degree greater than $2$. 
\subsection{$6 \leq p < 14$}
We first prove \eqref{eq:main_thm_degree3} for $C^2\leq T < C^4$. 
As in the proof of Theorem \ref{theorem:long_time_Strchartz}, it suffices to
show
\begin{equation*}
  \|e^{2\pi i tD^{2}_{x}}f(x)\|_{L^6_t \mathcal{L}^6_{x}([0, C^2] \times\R)}
  \lesssim_{\varepsilon}C^{\frac 23 + \varepsilon}\|f(x)\|_{\mathcal{L}^2_{x}(\R)}. 
\end{equation*}
The cases $p > 6$ follow from the interpolation between this estimate and 
the trivial case $p = \infty$. If $T > C^2$, we split the interval $[0,T]$ into
$T / C^2$ subintervals of length $C^2$ and apply the above estimate on each 
subinterval. 
From the same computation in the previous section, we have
\begin{align*}
  &\|e^{2\pi i tD^{2}_{x}}R_{C}f(x)\|^6_{L^{6}_{t}\mathcal{L}^6_x}\\
  &\;= \Bigl|\sum_{\lambda_1 + \lambda_3 + \lambda_5= \lambda_2 + \lambda_4 + \lambda_6}
  \int^{C^2}_{0}
  e^{2\pi i\{(|\lambda_1|^2 + |\lambda_3|^2 + |\lambda_5|^2) - (|\lambda_2|^2 + |\lambda_4|^2 + |\lambda_6|^2)\}t}
  a_{\lambda_1}a_{\lambda_3}a_{\lambda_5} 
  \overline{a_{\lambda_2} a_{\lambda_4} a_{\lambda_6}}dt\Bigr|. 
\end{align*}
We first settle the case when $|(\lambda^2_1 + \lambda^2_3 + \lambda^2_5)
- (\lambda^2_2 + \lambda^2_4 + \lambda^2_6)| <  C^{-2}$. Set 
\[\Lambda_{0} := 
\{(\lambda_1,\lambda_2, \lambda_3, \lambda_4, \lambda_5, \lambda_6)\;; \; 
\substack{\lambda_1 + \lambda_3 + \lambda_5 = \lambda_2 + \lambda_4 + \lambda_6 \\
|(\lambda^2_1 + \lambda^2_3 + \lambda^2_5)
- (\lambda^2_2 + \lambda^2_4 + \lambda^2_6)| <  C^{-2}}\}. \]
By the AM-GM inequality and the Cauchy--Schwarz, it holds that
\begin{align*}
  &\Bigl|\sum_{\lambda_i \in \Lambda_{0}}
  \int^{C^2}_{0}
  e^{2\pi i\{(|\lambda_1|^2 + |\lambda_3|^2 + |\lambda_5|^2) - (|\lambda_2|^2 + |\lambda_4|^2 + |\lambda_6|^2)\}t}
  a_{\lambda_1}a_{\lambda_3}a_{\lambda_5} 
  \overline{a_{\lambda_2} a_{\lambda_4} a_{\lambda_6}}dt\Bigr|\\
&\leq C^{2}\sum_{\lambda_i \in \Lambda_{0}}
|a_{\lambda_1}a_{\lambda_3}a_{\lambda_5}a_{\lambda_2}a_{\lambda_4}a_{\lambda_6}|\\
&\lesssim C^{2}\Bigl(\sum_{\lambda_i\in \Lambda_0}
|a_{\lambda_1}a_{\lambda_3}a_{\lambda_5}|^2 + \sum_{\lambda_i\in \Lambda_0}
|a_{\lambda_2}a_{\lambda_4}a_{\lambda_6}|^2\Bigr). 
\end{align*}
By symmetry, it is enough to consider the first term and we find that
\[\sum_{\lambda_i\in \Lambda_0}
|a_{\lambda_1}a_{\lambda_3}a_{\lambda_5}|^2
\leq \sup_{\lambda_1, \lambda_3,\lambda_5}
\#\{(\lambda_2, \lambda_4, \lambda_6)\; ; \; 
\substack{\lambda_1 + \lambda_3 + \lambda_5 = \lambda_2 + \lambda_4 + \lambda_6 \\
|(\lambda^2_1 + \lambda^2_3 + \lambda^2_5)
- (\lambda^2_2 + \lambda^2_4 + \lambda^2_6)| <  C^{-2}}\}
\Bigl(\sum_{\lambda}|a_{\lambda}|^2\Bigr)^3. \]
Note that since $\omega$ is of degree greater than $2$, $(1,\omega, \omega^2)$
are linearly independent over $\Q$. Hence, we are led to show
\[\#\{(\lambda_2, \lambda_4, \lambda_6)\; ; \; 
\substack{\lambda_2 + \lambda_4 + \lambda_6 = M \\
|(\lambda^2_2 + \lambda^2_4 + \lambda^2_6) - S| <  C^{-2}}\}
\lesssim_{\varepsilon} C^{2 + \varepsilon}\]
for fixed $M\in \Z + \omega\Z$ with $|M|\lesssim C$ and
fixed $S\in \Z + \omega\Z + \omega^2 \Z$ with $|S| \lesssim C^2$. Similar to the 
proof of Theorem \ref{theorem:long_time_Strchartz}, it suffices to show
\[\#\{t = t_1 + \omega t_2 + \omega^2 t_{3}\; ; \;
|t - S| < C^{-2}, |t_1|, |t_2|, |t_3| \lesssim C^2\} \lesssim_{\varepsilon}
C^{2 + \varepsilon}\]
For two elements of this set $t_{1} + \omega t_2 + \omega^2 t_{3}, 
s_{1} + \omega s_{2} + \omega^2 s_{3}$, set $m_i := t_i - s_i$ $(i = 1,2,3)$. 
Then we have
\[|m_1 + \omega m_2 + \omega^2 m_3| \lesssim C^{-2}. \]
From this inequality and Theorem \ref{theorem:multi_roth}, we obtain 
\[\max\{|m_1|, |m_2|, |m_3|\} \gtrsim_{\varepsilon} C^{\frac{2}{2 + \varepsilon}}. \]
Thus, the set $\{t = t_1 + \omega t_2 + \omega^2 t_{3}\; ; \;
|t - S| < C^{-2}, |t_1|, |t_2|, |t_3| \lesssim C^2\}$ is $C^{\frac{2}{2 + \varepsilon}}$-separated. Noting that this set is covered by a rectangle with dimensions $1\times C^2 \times C^2$, we obtain
\[\lesssim_{\varepsilon}C^{\frac{2 + 2\varepsilon}{2 + \varepsilon}}
C^{\frac{2 + 2\varepsilon}{2 + \varepsilon}} \leq C^{2 + 2\varepsilon}. \]
This finishes the proof of the case when $|(\lambda^2_1 + \lambda^2_3 + \lambda^2_5)
- (\lambda^2_2 + \lambda^2_4 + \lambda^2_6)| <  C^{-2}$. We next consider the case 
when $|(\lambda^2_1 + \lambda^2_3 + \lambda^2_5)
- (\lambda^2_2 + \lambda^2_4 + \lambda^2_6)| \geq  C^{-2}$. As in the previous 
proof, for a dyadic number $A \in (C^{-2}, 100C^2]$, we define a set 
$\Lambda_{\omega, C,A}$ as
\[\Lambda_{\omega, C,A} := \Bigl\{\lambda_1,\cdots, \lambda_6\in \Lambda_{\omega,C}
\; ; \; \substack{\lambda_1 + \lambda_3 + \lambda_5 = \lambda_2 + \lambda_4 + \lambda_6,\\
A \le |(|\lambda_1|^2 + |\lambda_3|^2 + |\lambda_5|^2) - 
(|\lambda_2|^2 + |\lambda_4|^2 + |\lambda_6|^2)| < 2A}\Bigr\}\]
and we have the bound
\begin{align*}
  &\int^{C^2}_{0}\sum_{\lambda_i\in \Lambda_{\omega,C,A}}e^{2\pi i\{(|\lambda_1|^2 + |\lambda_3|^2 + |\lambda_5|^2) - (|\lambda_2|^2 + |\lambda_4|^2 + |\lambda_6|^2)\}t}
  a_{\lambda_1}a_{\lambda_3}a_{\lambda_5} 
  \overline{a_{\lambda_2} a_{\lambda_4} a_{\lambda_6}}dt\\
  &\lesssim A^{-1}\sum_{\lambda_i\in \Lambda_{\omega,C,A}}
  |a_{\lambda_1}a_{\lambda_3}a_{\lambda_5}
  a_{\lambda_2}a_{\lambda_4}a_{\lambda_6}|. 
\end{align*}
So, we are led to show
\[\sum_{\lambda_i\in \Lambda_{\omega,C,A}}
|a_{\lambda_1}a_{\lambda_3}a_{\lambda_5}
a_{\lambda_2}a_{\lambda_4}a_{\lambda_6}|
\lesssim_{\varepsilon} A C^{4+ \varepsilon} 
\Bigl(\sum_{\lambda}|a_{\lambda}|^2\Bigr)^3. \]
We further decompose the sets $\Lambda_{\omega,C,A}$ into the following subsets
\begin{equation*}
  \Sigma_{X}:= \{\lambda_1, \cdots, \lambda_6\in 
  \Lambda_{\omega, C,A}\; ;\; 
  \lambda^2_1 + \lambda^2_3 + \lambda^2_5
  -\lambda^2_2 - \lambda^2_4 - \lambda^2_6= X\}
\end{equation*} and consider
\[\sum_{\substack{X\in \Z + \omega\Z + \omega^2 \Z
\\ A\leq X < 2A}}\sum_{\lambda_i\in \Sigma_{X}}
|a_{\lambda_1}a_{\lambda_3}a_{\lambda_5}
a_{\lambda_2}a_{\lambda_4}a_{\lambda_6}| 
\lesssim_{\varepsilon} A C^{4 + \varepsilon}\Bigl(\sum_{\lambda}|a_{\lambda}|^2\Bigr)^3. \]
Therefore, it suffices to show the following two claims:
\begin{itemize}
  \item \begin{equation*}
  \sum_{\substack{X\in \Z + \omega\Z + \omega^2\Z\\ 
  A \leq X < 2A}}1 \lesssim_{\varepsilon}A C^{4 + \varepsilon}
  \end{equation*}

  \item For each $X\in \Z + \omega\Z + \omega^2 \Z$ with $A \leq X < 2A$, 
  \begin{equation*}
    \sum_{\lambda_i\in \Sigma_{X}}
    |a_{\lambda_1}a_{\lambda_3}a_{\lambda_5}
    a_{\lambda_2}a_{\lambda_4}a_{\lambda_6}| 
    \lesssim_{\varepsilon} C^{\varepsilon}\Bigl(\sum_{\lambda}|a_{\lambda}|^2\Bigr)^3
\end{equation*}
\end{itemize}
These inequalities follow from Theorem \ref{theorem:multi_roth} and Lemma \ref{lemma:additive_energy} as in the previous cases. So, we present them briefly here. First, for 
\[
  \sum_{\substack{X\in \Z + \omega\Z + \omega^2\Z\\ 
  A \leq X < 2A}}1 \lesssim_{\varepsilon}A C^{4 + \varepsilon}, 
\]
we need to consider two subcases: (a) $A \geq 1$ and (b) $A < 1$. The case (a) is easy because if $X = t_1 + \omega t_2 + \omega^{2}t_3$, we choose $t_2, t_3$ from $O(C^{2}\times C^{2})$ choices and then choose $t_1$ from $O(A)$ choices. For the case (b), from Theorem~\ref{theorem:multi_roth} with a similar argument as in the previous case, we find that the set $\{X = t_1 + \omega t_2 + \omega^2 t_{3}\; ; \;
X\sim A, |t_1|, |t_2|, |t_3| \lesssim C^2\}$ is $A^{\frac{1}{2 + \varepsilon}}$-separated. Noting that this set is covered by a rectangle with dimensions $1\times C^2 \times C^2$, we obtain
\[\lesssim_{\varepsilon}A^{\frac{1}{2 + \varepsilon}}C^{2}
A^{\frac{1}{2 + \varepsilon}}C^{2} \lesssim A^{\frac{2}{2 + \varepsilon}}C^{4}. \]

For 
\begin{equation*}
    \sum_{\lambda_i\in \Sigma_{X}}
    |a_{\lambda_1}a_{\lambda_3}a_{\lambda_5}
    a_{\lambda_2}a_{\lambda_4}a_{\lambda_6}| 
    \lesssim_{\varepsilon} C^{\varepsilon}\Bigl(\sum_{\lambda}|a_{\lambda}|^2\Bigr)^3, 
\end{equation*}
applying the AM-GM inequality and the Cauchy--Schwarz, the left-hand side is bounded by
\[\Bigl(\sum_{\lambda} |a_{\lambda}|^2\Bigr)^3 
\sup_{\lambda_1, \lambda_3, \lambda_5}\#\Bigl\{ (\lambda_2, \lambda_4, \lambda_6)\; ;\; \substack{\lambda_1 + \lambda_3 + \lambda_5 = \lambda_2 + \lambda_4 + \lambda_6\\
  |(|\lambda_1|^2 + |\lambda_3|^2 + |\lambda_5|^2) 
- (|\lambda_2|^2 + |\lambda_4|^2 + |\lambda_6|^2)| = X} \Bigr\}
\]
So, we need to show 
\[\#\Bigl\{(\lambda_2, \lambda_4, \lambda_6)\; ;\; \substack{\lambda_2 + \lambda_4 + \lambda_6 = M\\
  \lambda^2_2 + \lambda^2_4 + \lambda^2_6 = S}\Bigr\}
  \lesssim_{\varepsilon}C^{\varepsilon}. 
\]
This immediately follows from Lemma~\ref{lemma:additive_energy}. 

Next, we present the proof of \eqref{eq:main_thm_degree3} for $T \geq C^4$. In this case, we will show
\begin{equation*}
  \|e^{2\pi i tD^{2}_{x}}f(x)\|_{L^6_t \mathcal{L}^6_{x}([0, C^4] \times\R)}
  \lesssim_{\varepsilon}C^{\frac 23 + \varepsilon}\|f(x)\|_{\mathcal{L}^2_{x}(\R)}. 
\end{equation*}
We also have 
\begin{align*}
  &\|e^{2\pi i tD^{2}_{x}}R_{C}f(x)\|^6_{L^{6}_{t}\mathcal{L}^6_x}\\
  &\;= \Bigl|\sum_{\lambda_1 + \lambda_3 + \lambda_5= \lambda_2 + \lambda_4 + \lambda_6}
  \int^{C^4}_{0}
  e^{2\pi i\{(|\lambda_1|^2 + |\lambda_3|^2 + |\lambda_5|^2) - (|\lambda_2|^2 + |\lambda_4|^2 + |\lambda_6|^2)\}t}
  a_{\lambda_1}a_{\lambda_3}a_{\lambda_5} 
  \overline{a_{\lambda_2} a_{\lambda_4} a_{\lambda_6}}dt\Bigr|. 
\end{align*}
As in the previous case, we consider two subcases (a) $|(\lambda^2_1 + \lambda^2_3 + \lambda^2_5)
- (\lambda^2_2 + \lambda^2_4 + \lambda^2_6)| <  C^{-4}$ and (b) $|(\lambda^2_1 + \lambda^2_3 + \lambda^2_5)
- (\lambda^2_2 + \lambda^2_4 + \lambda^2_6)| \geq  C^{-4}$. For (a), the proof is
also reduced to establish the bound 
\[\#\{t = t_1 + \omega t_2 + \omega^2 t_{3}\; ; \;
|t - S| < C^{-4}, |t_1|, |t_2|, |t_3| \lesssim C^2\} \lesssim_{\varepsilon}
C^{\varepsilon}. \]
For two elements of this set $t_{1} + \omega t_2 + \omega^2 t_{3}, 
s_{1} + \omega s_{2} + \omega^2 s_{3}$, set $m_i := t_i - s_i$ $(i = 1,2,3)$. 

Then we have
\[|m_1 + \omega m_2 + \omega^2 m_3| \lesssim C^{-4}. \]
From this inequality and Theorem \ref{theorem:multi_roth}, we obtain 
\[\max\{|m_1|, |m_2|, |m_3|\} \gtrsim_{\varepsilon} C^{\frac{4}{2 + \varepsilon}}. \]
and find that the set $\{t = t_1 + \omega t_2 + \omega^2 t_{3}\; ; \;
|t - S| < C^{-4}, |t_1|, |t_2|, |t_3| \lesssim C^2\}$ is $C^{\frac{4}{2 + \varepsilon}}$-separated. Noting that this set is covered by a rectangle with dimensions $1\times C^2 \times C^2$, we obtain
\[\lesssim_{\varepsilon}C^{\frac{2\varepsilon}{2 + \varepsilon}}
C^{\frac{2\varepsilon}{2 + \varepsilon}} \leq C^{2\varepsilon}. \]

For (b), we define similar decompositions $\Lambda_{\omega, C,A}$ and $\Sigma_{X}$ as before for $A\in (C^{-4}, 100 C^2)$ and $A \leq X < 2A$. Then we reduce the estimate to
the following two claims:
\begin{itemize}
  \item \begin{equation*}
  \sum_{\substack{X\in \Z + \omega\Z + \omega^2\Z\\ 
  A \leq X < 2A}}1 \lesssim_{\varepsilon}A C^{4 + \varepsilon}. 
  \end{equation*}

  \item For each $X\in \Z + \omega\Z + \omega^2 \Z$ with $A \leq X < 2A$, 
  \begin{equation*}
    \sum_{\lambda_i\in \Sigma_{X}}
    |a_{\lambda_1}a_{\lambda_3}a_{\lambda_5}
    a_{\lambda_2}a_{\lambda_4}a_{\lambda_6}| 
    \lesssim_{\varepsilon} C^{\varepsilon}\Bigl(\sum_{\lambda}|a_{\lambda}|^2\Bigr)^3. 
\end{equation*}
\end{itemize}
Their proofs are the same as those in the case when $C^{2}\leq T < C^{4}$, so we omit their details.

\subsection{$p \geq 14$}
In this subsection, we prove \eqref{eq:main_thm_degree3_l14} by following the idea in \cite{Degn_Germain_Guth_Myerson}. 
By interpolating with $p = \infty$, it suffices to show the case when $p = 14$. 
The main ingredient to prove this case is counting the number of integer solutions to a quadratic Parsell--Vinogradov system. For our purpose, we consider the following system: For fixed integers $s\geq 1$, $N\geq 1$, and
$a,b,\alpha, \beta, \gamma\in \Z$, 
\begin{align*}
  X_{1} + \cdots + X_{s} -(X_{s + 1} + \cdots + X_{2s})&= a,\\
  Y_{1} + \cdots + Y_{s} -(Y_{s + 1} + \cdots + Y_{2s})&= b,\\
  X^2_{1} + \cdots + X^2_{s} -(X^2_{s + 1} + \cdots + X^2_{2s})&= \alpha,\\
  X_{1}Y_{1} + \cdots + X_{s}Y_{s} -(X_{s + 1}Y_{s + 1} + \cdots +X_{2s}Y_{2s})&= \beta,\\
  Y^2_{1} + \cdots + Y^2_{s} -(Y^2_{s + 1} + \cdots + Y^2_{2s})&=\gamma
\end{align*}
with $-N \leq X_{i}, Y_{i}\leq N$. Let $\Omega^{s,N}_{a,b,\alpha,\beta,\gamma}$ be the set of integer tuples satisfying $-N\leq X_{i}, Y_{i}\leq N$ which satisfy the above system for given parameters $s,N,a,b,\alpha, \beta$, and $\gamma$ and let $J_{s}(N,S)$ denote the number of integer solutions to this system inside the subset $S\subset[-N, N]^2$. In Bourgain--Demeter~\cite{Bourgain_Demeter_Weyl_sum}, the decoupling estimate for two-dimensional manifold 
\[\{(t,s,t^2, ts, s^2)\; ;\; (t,s)\in [0,1]^2\}\subset\R^{5}\]
was established. By using the decoupling estimate with a slightly modified argument in the proof of Corollary 2.3 in \cite{Bourgain_Demeter_Weyl_sum}, we obtain the bound for $J_{s}(N,S)$: 
\begin{theorem}[Theorem 2.1 in \cite{Bourgain_Demeter_Weyl_sum}]
  Let $s\geq 1$, $N \geq 1$, and $a,b,\alpha,\beta,\gamma\in \Z$. Let $S\subset [-N,N]^2$ be any subset. Then the estimate 
  \begin{equation}
    J_{s}(N,S)\lesssim_{\varepsilon,s} N^{\varepsilon}(N^{2s - 2} + N^{4s - 10})|S|
    \label{eq:Bourgain_Demeter_PV}
  \end{equation}
  holds. 
\end{theorem}

To apply this result to our Strichartz estimate, we need the following well-known reduction:
\begin{lemma}[Lemma 3.1 in \cite{reversing}]
  Let $p\in[1,\infty)$, $c\in (0,\infty)$, and let $T:\C^{N}\to [0,\infty)$ be a 
  sublinear function such that 
  \begin{equation}
    T(1_{S})\leq C \|1_{S}\|_{\ell^p}\quad (\forall S \subset\{1,2,\cdots, N\}). 
    \label{eq:lem_reduction_assump}
  \end{equation}
  Then 
  \begin{equation*}
    T(a)\leq C' (1 + (\log N)^{\frac{1}{p'}}/p)\|a\|_{\ell^p}
  \end{equation*}
  holds for all $a = (a_{n})_{n\in\{1,2,\cdots,N\}}\subset\C^{N}$, where 
  \[C' = 2^{\frac 1p}4^\frac{1}{p'}C. \]
  \label{lem:reduction_to_indicator}
\end{lemma}
The proof of this lemma relies on the theory of the Lorentz spaces, that is, 
the estimate \eqref{eq:lem_reduction_assump} implies 
\[
  T(a)\leq C\|a\|_{\ell^{p,1}}\quad (\forall a = (a_{n})_{n\in\{1,2,\cdots,N\}}\subset [0,\infty))
\]
where 
\[\|a\|_{\ell^{p,1}}:=\int^{\infty}_{0}\#\{n\in\{1,2,\cdots, N\}\; ; \; a_{n}\geq s\}^{\frac 1p} ds\]
(for example, see Exercise 1.4.7. in \cite{Grafakos_classical}). Then we bound 
the $\ell^{p,1}$ norm by the $\ell^{p}$ with some subpolynomial factor with respect to $N$. Thus, combining Lemma~\ref{lemma:kronecker_weyl} with Lemma~\ref{lem:reduction_to_indicator}, our $L^{14}$-Strichartz estimate is reduced to 
\begin{equation}
  \|e^{2\pi i t\tilde{\Delta}}\mathcal{F}^{-1}[1_{S}]\|_{L^{14}_{t,x}([0,T]\times \T^{2})}\lesssim \|1_{S}\|_{\ell^2}\quad (\forall S\subset [-C,C]^{2}\cap \Z^{2})
  \label{eq:L14_indicator}
\end{equation}
where $\tilde{\Delta}:= \partial^2_{1} + 2\omega\partial_{1}\partial_{2} + \omega^{2}\partial^{2}_{2}$ and $\omega\in \R\backslash\Q$ is an algebraic number of degree greater than $2$. Now, we are ready to prove the $L^{14}$-Strichartz estimate. 
\begin{proof}[Proof of the $L^{14}$-Strichartz estimate]
  By multiplying out the left-hand side of \eqref{eq:L14_indicator}, we have 
  \[
  \|e^{2\pi it\tilde{\Delta}}\mathcal{F}^{-1}[1_{S}]\|^{14}_{L^{14}_{t,x}}
  \lesssim \sum_{\alpha, \beta,\gamma\in [-C^{2}, C^{2}]\cap\Z}\sum_{(k_i, n_i)\in S^{14}\cap \Omega^{7, C}_{0,0,\alpha, \beta, \gamma}}\min\{|\alpha + 2\beta\omega + \gamma\omega^{2}|^{-1}, T\}. 
  \]
  We first consider the case when $T > |\alpha + 2\beta\omega + \gamma\omega^2|^{-1}$. Then we need to estimate
  \[\sum_{\alpha, \beta,\gamma\in [-C^{2}, C^{2}]\cap\Z}\sum_{(k_i, n_i)\in S^{14}\cap \Omega^{7, C}_{0,0,\alpha, \beta, \gamma}}|\alpha + 2\beta\omega + \gamma\omega^2|^{-1}. \]
  Splitting $|\alpha + 2 \beta\omega + \gamma\omega^2|$ into dyadic scales, this summation is bounded by 
  \[
  \sum_{T > 2^{-j}}\sum_{|\alpha + 2\beta\omega + \gamma \omega^2| \sim 2^{j}}2^{-j}\sum_{(k_i, n_i)\in S^{14}\cap \Omega^{7, C}_{0,0,\alpha, \beta, \gamma}}1.
  \]
  For the innermost summation, we have 
  \[
  \sum_{(k_i, n_i)\in S^{14}\cap \Omega^{7, C}_{0,0,\alpha, \beta, \gamma}}1 = \#(S^{14}\cap \Omega^{7, C}_{0,0,\alpha, \beta, \gamma}) 
  \lesssim_{\varepsilon} C^{6}|S|^{7}. 
  \]
  Indeed, we first fix six variables of $(n_{i}, k_{i})$, giving $|S|^6$ choices. We then apply \eqref{eq:Bourgain_Demeter_PV} with $s = 4$ and get 
  \[J_{4}(C, S) \lesssim_{\varepsilon} C^{6}|S|. \]
  Next, for the intermediate summation $\sum_{|\alpha + 2\beta\omega + \gamma\omega^2|\sim 2^{j}}$, by using Theorem~\ref{theorem:multi_roth},  
  \[
  \sum_{|\alpha + 2\beta\omega + \gamma\omega|\sim 2^{j}}1\lesssim_{\varepsilon} C^{4+ \varepsilon}2^{j} + 1. 
  \]
  In summary, it holds that 
  \[
  \sum_{T > 2^{-j}}\sum_{|\alpha + 2\beta\omega + \gamma \omega^2| \sim 2^{j}}2^{-j}\sum_{(k_i, n_i)\in S^{14}\cap \Omega^{7, C}_{0,0,\alpha, \beta, \gamma}}\lesssim \sum_{T> 2^{-j}}2^{-j}(C^{4}2^{j} + 1)C^{6}|S|^{7}\lesssim_{\varepsilon} C^{\varepsilon}(C^{4} + T)C^{6}|S|^{7}. 
  \]
  For the case when $T < |\alpha + 2\beta\omega + \gamma\omega^2|^{-1}$, applying Theorem~\ref{theorem:multi_roth} with a similar argument as that in the proofs of Theorem~\ref{theorem:long_time_Strchartz} and Theorem~\ref{theorem:degree3}, we have 
  \[
  \sum_{\substack{\alpha, \beta,\gamma\in [-C^{2}, C^{2}]\cap\Z\\T < |\alpha + 2\beta\omega + \gamma\omega^2|^{-1}}}\lesssim_{\varepsilon}T^{-1}C^{4 + \varepsilon} + 1
  \]
  By combining this with $\sum_{(k_i, n_i)\in S^{14}\cap \Omega^{7, C}_{0,0,\alpha, \beta, \gamma}}1 = \#(S^{14}\cap \Omega^{7, C}_{0,0,\alpha, \beta, \gamma}) 
  \lesssim_{\varepsilon} C^{6}|S|^{7}$, we obtain the desired estimate. 

  From Lemma~\ref{lem:reduction_to_indicator} and Lemma~\ref{lemma:kronecker_weyl}, we conclude 
  \[
  \|e^{2\pi i tD^{2}_{x}}R_{C}f(x)\|_{L^{14}_{t}\mathcal{L}^{14}_{x}([0,T]\times\R)}\lesssim_{\varepsilon} C^{\varepsilon}(C^{\frac 57} + T^{\frac{1}{14}}C^{\frac 37})\|R_{C}f(x)\|_{\mathcal{L}^2_{x}(\R)}. 
  \]
\end{proof}

\section{Endpoint $L^4$ estimate}
In this subsection, we give the proof of Theorem~\ref{thm:endpoint_schippa}. 
The key tool for the proof is the reverse square function estimate for the 
one-dimensional curve $(\xi, |\xi|^a)$ $a\in (0,1)\cup(1,\infty)$ which was 
proved in the author's joint work with Bulj and Shiraki \cite{author2026reverse}. 

\begin{proposition}
  Let $a\in (0,1)\cup (1,\infty)$ and $\delta \in (0,1)$.
  For any Schwartz function
  $F\in \mathcal{S}(\R^2)$ with 
  \begin{equation*}
    \supp\widehat{F}\subset 
    \{(\xi_1, \xi_2)\in\R^2\; ; \; |\xi_1| \leq 1, \; 
    |\xi_2 - |\xi_1|^a| < \delta\}, 
  \end{equation*}
  the following estimate holds
  \begin{equation}
    \|F\|_{L^4 (\R^2)} \lesssim 
    \Bigl\|\Bigl(\sum_{k\in\Z}|F_{\theta_k}|^2\Bigr)^{\frac 12}\Bigr\|_{L^4 (\R^2)}
  \end{equation}
  where 
  \begin{align*}
    \widehat{F}_{\theta_k}(\xi):=
    \begin{cases}
      1_{[\delta^{\frac 1a}(k - \frac 12), \delta^{\frac 1a}(k + \frac 12)]\times\R}(\xi)
      \widehat{F}(\xi) \quad &(a \in [2, \infty)),\\
      1_{[\delta^{\frac 12}(k - \frac 12), \delta^{\frac 12}(k + \frac 12)]\times\R}(\xi)
      \widehat{F}(\xi) \quad &(a \in (0,1)\cup (1,2)). 
    \end{cases}
  \end{align*}
  \label{prop:reverse_square_gen}
\end{proposition}

From this proposition, we will prove our Strichartz estimates. 
Since we apply the reverse square function estimate
here, unlike Theorem \ref{theorem:long_time_Strchartz} 
and Theorem \ref{theorem:degree3},
we can easily show cases with $\nu > 2$. 
\begin{proposition}
  Let $2 \leq \nu \in \N$, $\vec{\omega} = (\omega_1, \omega_2, \cdots , \omega_\nu)\in \R^{\nu}$ 
  be non-resonant, $C \geq 1$, 
  and $N \lesssim_{\vec{\omega}} C$.
  Then the following estimate holds: 
  \begin{equation}
    \|e^{2\pi i t|D_{x}|^{a}}P_{N}R_C f(x)\|_{L^4_{t}\mathcal{L}^{4}_{x}([0,1] \times \R)} 
    \lesssim C^{\frac{\nu - 1}{4}}(1 + N^{\sigma_a})\|P_{N}R_{C}f(x)\|_{\mathcal{L}^2_x (\R)}\label{eq:L4fractional}
  \end{equation}
  where 
  \begin{align*}
    \sigma_a = 
    \begin{cases}
      0 \quad &(a \geq 2),\\
      \frac{2 - a}{8} &(a\in (0,1)\cup (1,2)). 
    \end{cases}
  \end{align*}
  \label{proposition:fractional_schrodinger}
\end{proposition}

\begin{proof}
  Let $\eta \in \mathcal{S}(\R)$ satisfy the following two conditions:
  \begin{align*}
    \begin{cases}
      \supp\mathcal{F}_{\R}[\eta]\subset [-1,1], \\
      \eta \gtrsim 1 \; \mathrm{on}\; [-1,1].
    \end{cases}
  \end{align*}
  Fix $L > N^{a-1}$. We consider 
  \[\frac{1}{2L}\int^{1}_{0}\int^{L}_{-L}|e^{2\pi i t|D_{x}|^{a}}
  P_{N}R_{C}f(x)
  |^4 dx dt. \]
  Rescaling as $x \to N^{-1}x$ and $t \to N^{-a}t$, we find that the above integral
  is equal to 
  \[\frac{1}{2L N^{1 + a}}\int^{N^a}_{0}\int^{NL}_{-NL}
  |e^{2\pi i t|D_{x}|^{a}}g(x)
  |^4 dx dt\]
  where $g(x) = P_{N}R_{C}f(x/N)$ with $\supp\widehat{g}\subset [-1,1]$. 

  Let $\{I_{N^a}\}$ be a family of intervals of length $N^a$ with finite overlap 
  that covers the interval $[-LN , LN]$. Then the above integral 
  is bounded as 
  \begin{align*}
    &\frac{1}{2L N^{1 + a}}\int^{N^a}_{0}\int^{NL}_{-NL}
    |e^{2\pi i t|D_{x}|^{a}}g(x)
    |^4 dx dt\\
    \lesssim &\frac{1}{2LN^{1 + a}}\sum_{I_{N^a}}\int^{N^a}_{0}\int_{I_{N^a}}
    |e^{2\pi i t|D_{x}|^{a}}g(x)
    |^4 dx dt\\
    \lesssim &\frac{1}{2LN^{1 + a}}\sum_{I_{N^a}}\int_{\R^2}
    |e^{2\pi i t|D_{x}|^{a}}g(x)
    \eta_{I_{N^a}}(x)\eta_{[0,N^a]}(t)|^4 dx dt
  \end{align*}
  where $\eta_{I_{N^a}}$ are compositions of $\eta\in \mathcal{S}(\R)$ and  
  affine transformations so that 
  $\supp\mathcal{F}_{\R}[\eta_{I_{N^a}}]\subset [-N^{-a}, N^{-a}]$ 
  and $\eta_{I_{N^a}}\gtrsim 1$ on $I_{N^a}$. Noting
  \[e^{2\pi i t|D_{x}|^{a}}g(x)
     = \sum_{\substack{k\in \Z^{\nu}\\
    |k| \leq C, \; |\vec{\omega}\cdot k|\leq N}}
    e^{2\pi i[\frac{(\vec{\omega}\cdot k)}{N}x + |\frac{\vec{\omega}\cdot k}{N}|^a t]
    }\widehat{g}(\vec{\omega}\cdot k), \]
  the Fourier support of the function $e^{2\pi i t|D_{x}|^{a}}g(x)
  \eta_{I_{N^a}}(x)\eta_{[0,N^a]}(t)$
  is contained in 
  \[\{(\xi_1, \xi_2)\in\R^2\; ; \; |\xi_1| \leq 1, \; 
  |\xi_2 - |\xi_1|^a| < cN^{-a}\}.\] 
  We first show the case when $a \geq 2$. Applying Proposition~\ref{prop:reverse_square_gen} with $a \geq 2$, 
  it holds that 
  \begin{align*}
    &\int_{\R^2}\Bigl|\sum_{\substack{k\in \Z^{\nu}\\
    |k| \leq C, \; |\vec{\omega}\cdot k|\leq N}}
    e^{2\pi i[\frac{(\vec{\omega}\cdot k)}{N}x + |\frac{\vec{\omega}\cdot k}{N}|^a t]
    }\widehat{g}(\vec{\omega}\cdot k)\eta_{I_{N^a}}(x)\eta_{[0,N^a]}(t)\Bigr|^4\\
    \lesssim&\int_{\R^2}\Bigl(\sum_{n\in \Z}\Bigl|
    \Bigl[\sum_{\substack{k\in \Z^{\nu}\\
    |k| \leq C, \; |\vec{\omega}\cdot k|\leq N}}
    e^{2\pi i[\frac{(\vec{\omega}\cdot k)}{N}x + |\frac{\vec{\omega}\cdot k}{N}|^a t]
    }\widehat{g}(\vec{\omega}\cdot k)\eta_{I_{N^a}}(x)\eta_{[0,N^a]}(t)\Bigr]_{\theta_n}\Bigr|^2\Bigr)^2
  \end{align*}
  where for $F\in \mathcal{S}(\R^2)$ and $n\in \Z$, we have 
  defined $\mathcal{F}_{\R^2}[F_{\theta_n}] = \mathcal{F}_{\R^2}[F]\cdot 
  1_{[N^{-1}(n - 1), N^{-1}(n + 1)]\times \R}$. The boundedness of the (vector valued)
  Hilbert transform reveals that 
  \begin{align*}
    &\int_{\R^2}\Bigl(\sum_{n\in \Z}\Bigl|
    \Bigl[\sum_{\substack{k\in \Z^{\nu}\\
    |k| \leq C, \; |\vec{\omega}\cdot k|\leq N}}
    e^{2\pi i[\frac{(\vec{\omega}\cdot k)}{N}x + |\frac{\vec{\omega}\cdot k}{N}|^a t]
    }\widehat{g}(\vec{\omega}\cdot k)\eta_{I_{N^a}}(x)\eta_{[0,N^a]}(t)\Bigr]_{\theta_n}\Bigr|^2\Bigr)^2dx\\
    =&
    \int_{\R^2}\Bigl(\sum_{n\in \Z}\Bigl|
    \Bigl[\sum_{\substack{k\in \Z^{\nu}\\
    |k| \leq C, \; |\vec{\omega}\cdot k|\leq N}}
    1_{[N^{-1}(n-2), N^{-1}(n + 2)]}\Bigl(\frac{\vec{\omega}\cdot k}{N}\Bigr)
    e^{2\pi i[\frac{(\vec{\omega}\cdot k)}{N}x + |\frac{\vec{\omega}\cdot k}{N}|^a t]
    }\widehat{g}(\vec{\omega}\cdot k)\eta_{I_{N^a}}(x)\eta_{[0,N^a]}(t)\Bigr]_{\theta_n}\Bigr|^2\Bigr)^2dx\\
    \lesssim &
    \int_{\R^2}\Bigl(\sum_{n\in \Z}\Bigl|
    \sum_{\substack{k\in \Z^{\nu}\\
    |k| \leq C, \; |\vec{\omega}\cdot k|\leq N}}
    1_{[N^{-1}(n-2), N^{-1}(n + 2)]}\Bigl(\frac{\vec{\omega}\cdot k}{N}\Bigr)
    e^{2\pi i[\frac{(\vec{\omega}\cdot k)}{N}x + |\frac{\vec{\omega}\cdot k}{N}|^a t]
    }\widehat{g}(\vec{\omega}\cdot k)\eta_{I_{N^a}}(x)\eta_{[0,N^a]}(t)\Bigr|^2\Bigr)^2dx\\
    \lesssim &
    \int_{\R^2}\Bigl(\sum_{n\in \Z}\Bigl|
    \sum_{\substack{k\in \Z^{\nu}\\
    |k| \leq C, \; |\vec{\omega}\cdot k|\leq N}}
    1_{[N^{-1}(n-1), N^{-1}(n + 1)]}\Bigl(\frac{\vec{\omega}\cdot k}{N}\Bigr)
    e^{2\pi i[\frac{(\vec{\omega}\cdot k)}{N}x + |\frac{\vec{\omega}\cdot k}{N}|^a t]
    }\widehat{g}(\vec{\omega}\cdot k)\eta_{I_{N^a}}(x)\eta_{[0,N^a]}(t)\Bigr|^2\Bigr)^2dx.
  \end{align*}
  Summing up the intervals $I_{N^a}$ and rescaling as $x \to Nx$ and $t \to N^a t$, 
  we obtain the following 
  \begin{align*}
    &\Bigl(\frac{1}{2L}\int_{\R}\int^{L}_{-L}|e^{2\pi i t|D_{x}|^{a}}
    P_{N}R_{C}f(x)\eta(t)|^4 dxdt\Bigr)^{\frac 14}
    \\
    \lesssim & \Bigl(\frac{1}{2L}\int_{\R}\int_{\R}\Bigl(\sum_{n\in \Z}\Bigl|
    \sum_{\substack{k\in \Z^{\nu}\\
    |k| \leq C, \; |\vec{\omega}\cdot k|\leq N}}
    1_{[n-1,n + 1]}\Bigl(\vec{\omega}\cdot k\Bigr)
    e^{2\pi i[(\vec{\omega}\cdot k)x + |\vec{\omega}\cdot k|^a t]
    }\widehat{P_{N}R_{C}f}(\vec{\omega}\cdot k)
    \eta_{[-L,L]}(x)\eta(t)\Bigr|^2\Bigr)^2dx dt\Bigr)^{\frac 14}
  \end{align*} 
  From Minkowski's inequality, the above is bounded by 
  \[\lesssim \Bigl(\sum_{n\in \Z}
    \Bigl(\frac{1}{2L}\int_{\R}\int_{\R}\Bigl|
    \sum_{\substack{k\in \Z^{\nu}\\
    |k| \leq C, \; |\vec{\omega}\cdot k|\leq N}}
    1_{[n-1,n + 1]}\Bigl(\vec{\omega}\cdot k\Bigr)
    e^{2\pi i[(\vec{\omega}\cdot k)x + |\vec{\omega}\cdot k|^a t]
    }\widehat{P_{N}R_{C}f}(\vec{\omega}\cdot k)
    \eta_{[-L,L]}(x)\Bigr|^4dxdt\Bigr)^{\frac 12}\Bigr)^{\frac 12}. \]
  By Lemma~\ref{lemma:weighted_lp} and taking the limit $L\to\infty$, we have
  \begin{equation}
    \|e^{2\pi i t|D_{x}|^{a}}P_{N}R_C f(x)\eta (t)\|_{L^4_{t}\mathcal{L}^{4}_{x}(\R \times \R)} 
    \lesssim
    \Bigl(\sum_{n\in \Z}\|e^{2\pi i t|D_{x}|^{a}}P_{I_n}R_C f(x)\eta(t)\|^2_{L^4_{t}\mathcal{L}^{4}_{x}([0,1] \times \R)}  \Bigr)^{\frac 12}
    \label{eq:pre_Strichartz}
  \end{equation}
  where $\widehat{P_{I_n}f}
  := 1_{\vec{\omega}\cdot k\in 
  [n - 1,n + 1]\cap [-N, N]}
  \widehat{f}$. 
  Hence, we have to count the number of lattice points in 
  \[\{k\in \Z^\nu \; ; \; |k|\sim C, \; 
  \vec{\omega}\cdot k\in [n - 1, n + 1]\}. \]
  Noting that this set is covered by a box with dimension $1 \times \underbrace{C \times \cdots \times C}_{\nu -1}$, we have the following
  \[\#\{k\in \Z^\nu \; ; \; |k|\sim C, \; 
  \vec{\omega}\cdot k\in [n - 1, n + 1]\}
  \lesssim_{\varepsilon}C^{\nu - 1} + 1. \]
  Applying Lemma~\ref{lemma:Bernstein}, we obtain
  \[\|e^{2\pi i t|D_{x}|^{a}}
    P_{I_n}R_C f(x)\eta(t)\|_{L^4_{t}\mathcal{L}^{4}_{x}(\R \times \R)} 
  \lesssim_{\varepsilon}C^{\frac{\nu - 1}{4}}
  \|P_{I_{n}}R_{C}f(x) \|_{\mathcal{L}^2_{x}}. \]
  Substituting this into \eqref{eq:pre_Strichartz} and applying Plancherel's theorem, 
  we obtain the desired estimates. 

  We now turn to the case when $a\in (0,1)\cup(1,2)$. The computation is almost the same as before. Applying Proposition~\ref{prop:reverse_square_gen} with $a\in (0,1)\cup(1,2)$, 
  it holds that 
  \begin{align*}
    &\int_{\R^2}\Bigl|\sum_{\substack{k\in \Z^{\nu}\\
    |k| \leq C, \; |\vec{\omega}\cdot k|\leq N}}
    e^{2\pi i[\frac{(\vec{\omega}\cdot k)}{N}x + |\frac{\vec{\omega}\cdot k}{N}|^a t]
    }\widehat{g}(\vec{\omega}\cdot k)\eta_{I_{N^a}}(x)\eta_{[0,N^a]}(t)\Bigr|^4\\
    \lesssim&\int_{\R^2}\Bigl(\sum_{n\in \Z}\Bigl|
    \Bigl[\sum_{\substack{k\in \Z^{\nu}\\
    |k| \leq C, \; |\vec{\omega}\cdot k|\leq N}}
    e^{2\pi i[\frac{(\vec{\omega}\cdot k)}{N}x + |\frac{\vec{\omega}\cdot k}{N}|^a t]
    }\widehat{g}(\vec{\omega}\cdot k)\eta_{I_{N^a}}(x)\eta_{[0,N^a]}(t)\Bigr]_{\theta_n}\Bigr|^2\Bigr)^2
  \end{align*}
  where for $F\in \mathcal{S}(\R^2)$ and $n\in \Z$, we have 
  defined $\mathcal{F}_{\R^2}[F_{\theta_n}] = \mathcal{F}_{\R^2}[F]\cdot 
  1_{[N^{-\frac 1a}(n - 1), N^{-\frac 1a}(n + 1)]\times \R}$. Then by the boundedness of the Hilbert transform, the rescaling $x\to Nx$, $t\to N^{a}$, and Lemma~\ref{lemma:weighted_lp}, we obtain 
  \begin{equation*}
    \|e^{2\pi i t|D_{x}|^{a}}P_{N}R_C f(x)\eta (t)\|_{L^4_{t}\mathcal{L}^{4}_{x}(\R \times \R)} 
    \lesssim
    \Bigl(\sum_{n\in \Z}\|e^{2\pi i t|D_{x}|^{a}}P_{I_n}R_C f(x)\eta(t)\|^2_{L^4_{t}\mathcal{L}^{4}_{x}([0,1] \times \R)}  \Bigr)^{\frac 12}
  \end{equation*}
  where $\widehat{P_{I_n}f}
  := 1_{\vec{\omega}\cdot k\in 
  [N^{1- \frac{a}{2}}(n - 1),N^{1- \frac{a}{2}}(n + 1)]\cap [-N, N]}
  \widehat{f}$. We need to count the number of lattice points in 
  \[\{k\in \Z^\nu \; ; \; |k|\sim C, \; 
  \vec{\omega}\cdot k\in [N^{1- \frac{a}{2}}(n - 1),N^{1- \frac{a}{2}}(n + 1)]\}. \]
  Noting that this set is covered by a box with dimensions $N^{1-\frac{a}{2}} \times \underbrace{C \times \cdots \times C}_{\nu -1}$, we have the following
  \[\#\{k\in \Z^\nu \; ; \; |k|\sim C, \; 
  \vec{\omega}\cdot k\in [N^{1- \frac{a}{2}}(n - 1),N^{1- \frac{a}{2}}(n + 1)]\}
  \lesssim_{\varepsilon}N^{1-\frac{a}{2}}C^{\nu - 1} + 1. \]
  Finally, applying the Bernstein type inequality (Lemma~\ref{lemma:Bernstein}), we obtain the desired estimates. 
\end{proof}

\begin{proof}[Proof of Theorem~\ref{thm:endpoint_schippa}]
  Now, we are ready to prove Theorem~\ref{thm:endpoint_schippa}. 
  Noting that $N \lesssim_{\vomega}C$, from Lemma~\ref{lem:LP_Schippa}, Proposition~\ref{proposition:fractional_schrodinger}
  becomes the following form: 
  \[
  \|e^{2\pi i t|D_{x}|^{a}}R_C f(x)\|_{L^4_{t}\mathcal{L}^{4}_{x}([0,1] \times \R)} 
    \lesssim C^{\frac{\nu - 1}{4} + \sigma_a}\|R_{C}f(x)\|_{\mathcal{L}^2_x (\R)}. 
  \]
  Then 
  \begin{align*}
    \|e^{2\pi i t|D_{x}|^{a}}f(x)\|_{L^4 _{t}\mathcal{L}^{4}_x([0,1]
  \times \R)} &\lesssim \Bigl(\sum_{C:\;\mathrm{dyadic}}\|e^{2\pi i t|D_{x}|^{a}}
  R_{C}f(x)\|^2_{L^4 _{t}\mathcal{L}^{4}_x([0,1]\times \R)}\Bigr)^{\frac 12}\\
  &\lesssim \Bigl(\sum_{C:\; \mathrm{dyadic}}C^{\frac{\nu -1}{2} + 2\sigma_a}\|R_{C}f(x)\|^{2}_{\mathcal{L}^2_{x}(\R)}\Bigr)^{\frac 12}\\
  &\sim \|f\|_{\mathcal{H}^{\frac{\nu - 1}{4} + \sigma_a}_{\vomega}(\R)}. 
  \end{align*}
  Finally, applying the density argument, we obtain the inequality \eqref{eq:endpoint_schippa}
  with $f\in \mathcal{H}^{s}_{\vomega} (\R)$ ($s \geq \frac{\nu - 1}{4} + \sigma_a$). 
\end{proof}

\begin{center}
    {\textsc{Declaration of}} {\large AI} {\textsc{Usage}}
    \par
\end{center}
The author used OpenAI's ChatGPT Plus to check portions of the proofs of Lemma~\ref{lem:littlewood_paley}, Lemma~\ref{lemma:weighted_lp}, and the Case 2-1 of Lemma~\ref{lemma:additive_energy}. ChatGPT Plus and DeepSeek were also used while the author was writing the manuscript for assistance with English proofreading, TeX syntax, and finding errors. The ideas, the proofs, and the actual writing of the manuscript are entirely the work of the author, who assumes full responsibility for the content. 

\bibliographystyle{abbrv}
\bibliography{ref}

\bigskip

\noindent\textit{Email}: inami@sustech.edu.cn

\noindent Department of Mathematics, Southern University of Science and Technology, Shenzhen, 518055, China

\end{document}